\documentclass{elsarticle}

\usepackage{hyperref}

\usepackage{graphicx}
\usepackage[latin1]{inputenc}

\input xy
\xyoption{all}

\usepackage{amsmath}
\usepackage{amsfonts}
\usepackage{amssymb}
\usepackage{amsthm}
\usepackage{indentfirst}
\usepackage{youngtab}
\usepackage{xcolor}
\usepackage{enumerate}

\usepackage[T1]{fontenc}

\newcommand{\comment}[1]{}

\newtheorem{Th}{Theorem}[section]

\newtheorem{Prop}[Th]{Proposition}
\newtheorem{Lem}[Th]{Lemma}
\newtheorem{Corol}[Th]{Corollary}
\newtheorem{Def}[Th]{Definition}

\theoremstyle{remark}
\newtheorem{Rem}[Th]{Remark}
\newtheorem{example}[Th]{Example}

\newcommand{\pnv}{}
\newcommand{\pnuv}{}
\newcommand{\BT}{\widetilde{ST}}
\newcommand{\mm}{\Psi}
\newcommand{\ex}[1]{(\tau_{\text{ex}}^{#1},\sigma_{\text{ex}}^{#1})}
\newcommand{\xx}{\mathbf{x}}
\newcommand{\vv}{\mathbf{v}}
\newcommand{\mon}{M}

\DeclareMathOperator{\Aut}{Aut}

\begin{document}
\author[labri]{Valentin F\'eray}\ead{feray@labri.fr}

\author[lix]{Ekaterina A. Vassilieva}\ead{ekaterina.vassilieva@lix.polytechnique.fr}

\address[labri]{LaBRI, Universit\'e Bordeaux 1, 351 cours de la lib\'eration, 33 400 Talence, France}
\address[lix]{LIX, Ecole Polytechnique, 91128, Palaiseau, France}

%\title[Linear coefficients of Kerov's polynomials]{Linear coefficients of Kerov's polynomials: bijective proof and refinement of a known result}
\title{Bijective enumeration of some colored permutations given by the product of two long cycles}
\begin{abstract}
%In this paper,
Let $\gamma_n$ be the permutation on $n$ symbols defined by $\gamma_n = (1\ 2\ \ldots\ n)$.
We are interested in an enumerative problem on colored permutations,
that is permutations $\beta$ of $n$ in which the numbers from $1$ to $n$
are colored with $p$ colors such that two elements in a same cycle have the same color.
We show that the proportion of colored permutations
such that $\gamma_n \beta^{-1}$ is a long cycle
is given by the very simple ratio $\frac{1}{n- p+1}$.
Our proof is bijective and uses combinatorial objects such as partitioned hypermaps and thorn trees.
This formula is actually equivalent to the proportionality of the number of long cycles $\alpha$ such that $\gamma_n\alpha$ has $m$ cycles and Stirling numbers of size $n+1$, an unexpected connection previously found by several authors by means of algebraic methods. Moreover, our bijection allows us to refine the latter result with the cycle type of the permutations.

%We look at the number of permutations $\beta$ of $[n]$ with $m$ cycles such that $(1\ 2\ \ldots\ n) \beta^{-1}$ is a long cycle. These numbers appear as coefficients of linear monomials in Kerov's and Stanley's character polynomials. A few authors, using algebraic methods, have found separately an unexpected connection with Stirling numbers of size $n+1$. We present the first combinatorial proof of this result, introducing a new bijection between partitioned maps and thorn trees. Moreover, we obtain a finer result, which takes the type of the permutations into account.
\end{abstract}

\begin{keyword}
Colored Permutations, Bipartite Maps, Long Cycle Factorization
\end{keyword}
\maketitle

%%%%%%%%%%%%%%%%%%%%%%%%%%%%%%%%%%%%%%%%%%%%%%%%
%%%%%%%%%      new section     %%%%%%%%%%%%%%%%%
%%%%%%%%%%%%%%%%%%%%%%%%%%%%%%%%%%%%%%%%%%%%%%%%

\section{Introduction}
The question of the number of factorizations of the long cycle $(1\ 2\ \ldots\ n)$ into two permutations with given number of cycles has already been studied via algebraic or combinatorial\footnote{It can be reformulated in terms of unicellular bipartite maps with given number of vertices, see paragraph \ref{subsect:partitioned_maps}.} methods \cite{%Jackson1988factorizations, 
Adrianov1998serie_carte_bicolore, SchaefferVassilieva:factorizations_long_cycle}.
In these papers, the authors obtain nice generating series for these numbers. Note that the combinatorial approach has been refined to state a result on the number of factorizations of the long cycle $(1\ 2\ \ldots\ n)$ in two permutations with given types \cite{MoralesVassilieva:factorizations_long_cycle}.\\

Unfortunately, even though generating series have nice compact forms,
the formulae for one single coefficient are much more complicated
(see for example \cite{GoupilSchaefferGenusExpansion}).
The case where one factor has to be also a long cycle is particularly interesting.
Indeed, the number $B'(n,m)$ of permutations $\beta$ of $[n]$ with $m$ cycles, such that $(1\ 2\ \ldots\ n) \beta^{-1}$ is a long cycle, is known to be the coefficient of some linear monomial in Kerov's and Stanley's character polynomials (see \cite[Theorem 6.1]{Biane2003Kerov} and \cite{Stanley2003rectangles,Feray:stanley_formula}). These polynomials express the character value of the irreducible representation of the symmetric group indexed by a Young diagram $\lambda$ on a cycle of fixed length in terms of some coordinates of $\lambda$.\\

The numbers $B'(n,m)$ admit a very compact formula in terms of Stirling numbers.
\begin{Th}[\cite{KwakLee1993}] \label{th:stanley}
Let $m \leq n$ be two positive integers with the same parity. Then
\begin{equation}\label{eq:stanley}
 \frac{n(n+1)}{2} B'(n,m) = s(n+1,m),
\end{equation}
where $s(n+1,m)$ is the unsigned Stirling number of the first kind,
that is the number of permutations of $[n+1]$ with $m$ cycles.
\end{Th}

This formula has been found independently by several authors:
J.H. Kwak and J. Lee \cite[Theorem 3]{KwakLee1993},
then D. Zagier \cite[Application 3]{ZagierCycles}
and finally R. Stanley \cite[Corollary 3.4]{Stanley2009product_cycles}.
Very recently, a combinatorial proof of this statement has been given
by R. Cori, M. Marcus and G. Schaeffer \cite{CoriMarcusSchaeffer2010}.
This paper is focused on an equivalent statement in terms of {\it colored} (or partitioned) permutations.

\begin{Def}\label{DefColoredPerm}
A colored permutation of $n$ with $p$ colors is a couple $(\beta,\varphi)$ where:
\begin{itemize}
    \item $\beta$ is a permutation of $n$;
    \item $\varphi$ is a surjective map from $\{1,\dots,n\}$ to a set $C$ of colors of cardinality $p$.
        We require that two elements belonging to the same cycle of $\beta$ have the same color.
\end{itemize}
In what follows, we consider that two colored permutations differing only by a bijection on the set of colors are the same object.
As such, coloration can be seen as a set partition of the set of cycles of $\beta$,
or as a set partition $\pi$ of  $\{1,\dots,n\}$ coarser than the set partition into cycles of $\beta$ (in other words, if $i$ and $j$ lie in the same cycle of $\beta$, they must be in the same part of $\pi$).
The set of colored permutations of $n$ with $p$ colors is denoted $\mathcal {C}(p,n)$.
\end{Def}
According to the last remark of definition \ref{DefColoredPerm}, we rather denote colored permutations $(\beta, \pi)$
where $\pi$ is a set partition coarser than the set partition into cycles of $\beta$.\\
These objects play an important role in the combinatorial study of the factorizations in the symmetric group,
as it is much easier to find direct bijections for colored factorizations than it is for classical ones
(see \cite{GouldenNica, GouldenSlofstra, Bernardi, SchaefferVassilieva:factorizations_long_cycle,
MoralesVassilieva:factorizations_long_cycle}).
Generating series of colored and classical factorizations are linked through simple formulae
(Lemma \ref{LemLinkOGF}).

We consider here an analogue of Theorem \ref{th:stanley} for colored permutations,
that is the problem of enumerating colored permutations such that
$(1\ 2\ \ldots\ n) \beta^{-1}$ is a long cycle.
We obtain the following elegant result:

\begin{Th}\label{th:reformulation}
Let $p \leq n$ be two positive integers.
 Choose randomly (with uniform probability) a colored permutation $(\beta,\pi)$ in $\mathcal {C}(p,n)$.
 Then the probability for $(1\ 2\ \ldots\ n) \beta^{-1}$ to be a long cycle is exactly $1/(n-p+1)$.
\end{Th}
 
 Given a colored permutation $(\beta,\pi)$ in $\mathcal {C}(p,n)$,
 the (unordered) sequence of the numbers of elements having the same color
 defines an integer partition of $n$ with $p$ parts, which we call the type of $(\beta,\pi)$.
 For any $\lambda$ integer partition of $n$,
 we note $\mathcal {C}(\lambda)$ the set of all colored permutations of type $\lambda$.
 Our main result is the following refinement of Theorem \ref{th:reformulation}:
 
 \begin{Th}[Main result]\label{th:reformulationfine}
     Let $p \leq n$ be two positive integers.
     Fix an integer partition $\lambda$ of size $n$ and length $p$.
 Choose randomly (with uniform probability) a colored permutation $(\beta,\pi)$
 in $\mathcal {C}(\lambda)$.
 Then the probability for $(1\ 2\ \ldots\ n) \beta^{-1}$ to be a long cycle
 is exactly $1/(n-p+1)$.
\end{Th}

%As former proofs of this result are purely algebraic,
%R. Stanley \cite{Stanley2009product_cycles} asked for
%a combinatorial proof of Theorem \ref{th:stanley}.
%This paper presents the first bijective approach proving this
%formula\footnote{To be comprehensive on the subject, let us mention the existence
%of a recent bijective proof of this result by R. Cori, M. Marcus and
%G. Schaeffer \cite{CoriMarcusSchaeffer2010}, which has been found after
%a first version of this paper had been released.}.
%In this paper, we present a new proof, which, even if it is not a direct bijection, is mainly combinatorial.\\
%In fact, 
In fact, counting colored permutations and counting permutations
without additional structure are two equivalent problems.
Therefore, one can deduce from Theorem \ref{th:reformulationfine}
a refinement of Theorem \ref{th:stanley}.
%Details are given in Section \ref{sect:reformulation}.

To state this new theorem, we need to introduce a few notations.
Recall that the type of a permutation is defined as the sequence of the lengths of its cycles,
sorted in increasing order.
With this notion, it is natural to refine the numbers $s(n+1,m)$ and $B'(n,m)$:
if $\lambda \vdash n$ (\textit{i.e.} $\lambda$ is a partition of $n$), let $A(\lambda)$ (resp. $B(\lambda)$) be the number of permutations $\beta \in S_n$ of type $\lambda$ (resp. with the additional condition that $(1\ 2\ \ldots\ n) \beta^{-1}$ is a long cycle).
Of course, $A(\lambda)$ is given by the simple formula $|\lambda|!/z_\lambda$,
where $m_i(\lambda)$ is the number of parts $i$ in $\lambda$ and $z_\mu=\prod_i i^{m_i(\mu)} m_i(\mu)!$.

Then, as Theorem \ref{th:stanley} deals with permutations of $[n]$ and $[n+1]$, we need operators on partitions which modify their size, but not their length. If $\mu$ (resp. $\lambda$) has at least one part $i+1$ (resp. $i$), let $\mu^{\downarrow (i+1)}$ (resp. $\lambda^{\uparrow (i)}$) be the partition obtained from $\mu$ (resp. $\lambda$) by erasing a part $i+1$ (resp. $i$) and adding a part $i$ (resp. $i+1$). For instance, using exponential notations (see \cite[chapter 1, section 1]{Macdo}), $(1^2 3^1 4^2)^{\downarrow (4)}=1^2 3^2 4^1$ and $(2^2 3^2 4)^{\uparrow (2)}=2^1 3^3 4^1$.\\
\begin{Th}[Corollary]\label{th:Stanleyfine}
    Let $m \leq n$ be two positive integers with the same parity.
 For each partition $\mu \vdash n+1$ of length $m$, one has:
\begin{equation}
\frac{n+1}{2} \sum_{\lambda = \mu^{\downarrow (i+1)}, i>0}  i\ m_i(\lambda)\ B(\pnv\lambda) = A(\mu) = \frac{(n+1)!}{z_\mu}.
\end{equation}
\end{Th}
From this result, one can immediately recover Theorem \ref{th:stanley} by summing
over all partitions $\mu$ of length $m$ and size $n+1$. Indeed,
\begin{multline*}
    \sum_{\mu \vdash n+1 \atop \ell(\mu)=m} \frac{n+1}{2}
    \sum_{\lambda = \mu^{\downarrow (i+1)}, i>0}
    i\ m_i(\lambda)\ B(\pnv\lambda)
    = \frac{n+1}{2} \sum_{\lambda \vdash n \atop \ell(\lambda)=m}
    \sum_{\mu = \lambda^{\uparrow (i)}, i>0} i\ m_i(\lambda)\ B(\pnv\lambda) \\
    = \frac{n+1}{2} \sum_{\lambda \vdash n \atop \ell(\lambda)=m} B(\pnv\lambda)
    \left(\sum_{i>0} i\ m_i(\lambda)\right) 
    = \frac{n(n+1)}{2} \sum_{\lambda \vdash n \atop \ell(\lambda)=m} B(\pnv\lambda)
    = \frac{n(n+1)}{2} B'(n,m).
\end{multline*}

To be comprehensive on the subject, we mention that
G. Boccara has found an integral formula for $B(\lambda)$
(see \cite{Boccara1980produit2cycles}), but there does not seem to be any direct link with our result.
\begin{Rem}\label{rem:syst_inv}
Theorem \ref{th:Stanleyfine}, written for all $\mu \vdash n+1$, gives the collection of numbers $B(\pnv\lambda)$ as solution of a sparse triangular system. Indeed, if we endow the set of partitions of $n$ with the lexicographic order, Theorem \ref{th:Stanleyfine}, written for the partition $\mu=(\lambda_1+1, \lambda_2,\lambda_3,\dots)$, gives $B(\pnv\lambda)$ in terms of the quantities $A(\mu)$ and $B(\pnv\nu)$ with $\nu > \lambda$.\bigskip
\end{Rem}

Note that the statement of Theorem \ref{th:reformulationfine} is much nicer than Theorem \ref{th:Stanleyfine}
(in particular, the fact that the ratio depends only on $|\lambda|$ and $\ell(\lambda)$ is quite surprising).
This suggests that it is interesting to work with colored permutations rather than 
with permutations without additional structure
(as it is done in \cite{CoriMarcusSchaeffer2010} for example).
%Our proof of Theorem \ref{th:reformulationfine} is entirely combinatorial,
%but is unfortunately not as simple as we could hope
%(in particular, it relies on a already non-trivial bijection of the second author and A. Morales
%\cite{MoralesVassilieva:factorizations_long_cycle}).
%It would be interesting to find a more simple explanation for Theorem \ref{th:reformulationfine}.
\medskip

\noindent {\em Outline of the paper.} Thanks to an interpretation of colored permutations
in terms of partitioned hypermaps (Section \ref{sect:hypermaps}),
we prove bijectively Theorem \ref{th:reformulationfine}
in Sections \ref{sect:def_psi}, \ref{sect:inj_psi} and \ref{sect:im_psi}.
Finally, in Section \ref{sect:reformulation},
we use algebraic computations in the ring of symmetric functions 
to show the equivalence with Theorem \ref{th:Stanleyfine}.
%As in the paper \cite{MoralesVassilieva:factorizations_long_cycle},
%the first step (section \ref{sect:reformulation}) of our proof of
%Theorem \ref{th:Stanleyfine} consists in a change of basis in the ring of symmetric
%functions in order to show the equivalence with the following statement:
%\begin{Th}\label{th:reformulation}
% Let $\lambda$ be a partition of $n$ of length $p$. Choose randomly (with uniform probability) a set-partition $\pi$ of $\{1,\ldots,n\}$ of type $\lambda$ and then (again with uniform probability) a permutation $\beta$ in $S_\pi$ (that means that each cycle of $\beta$ is contained in a block of $\pi$). Then the probability for $(1\ 2\ \ldots\ n) \beta^{-1}$ to be a long cycle is exactly $1/(n-p+1)$.
%\end{Th}
%Once again, such a simple formula is surprising.  
%We give a bijective proof in sections \ref{sect:def_psi}, \ref{sect:inj_psi} and \ref{sect:im_psi}. 
%Nevertheless, our proof is quite involved and it would be interesting to find a more direct bijection.
%Finally note that the theorem does not stand for a fixed set-partition $\pi$.

%%%%%%%%%%%%%%%%%%%%%%%%%%%%%%%%%%%%%%%%%%%%%%%%
%%%%%%%%%      new section     %%%%%%%%%%%%%%%%%
%%%%%%%%%%%%%%%%%%%%%%%%%%%%%%%%%%%%%%%%%%%%%%%%

\section{Combinatorial formulation of Theorem \ref{th:reformulationfine}}
\label{sect:hypermaps}

\subsection{Black-partitioned maps}\label{subsect:partitioned_maps}
By definition, a {\em map} is a graph drawn on a two-dimensional oriented closed compact surface (up to deformation), i.e. a graph with a cyclic order on the incident edges to each vertex. 
The faces of a map are the connected components of the surface without the graph (we require that these components are isomorphic to open discs).\\

As usual \cite{JacquesHypermaps, CoriHypermaps},
a couple of permutations $(\alpha,\beta)$ in $S_n$ can be represented as a bipartite map (or hypermap)
with $n$ edges labeled with integers from $1$ to $n$.
In this identification, $\alpha(i)$ (resp. $\beta(i)$) is the edge following $i$
when turning around its white (resp. black) extremity.
 White (resp. black) vertices correspond to cycles of $\alpha$ (resp. $\beta$).
 In this setting, faces of the map correspond to cycles of the product $\alpha \beta$.
 Hence, the condition $\alpha \beta = (1\ 2\ \ldots\ n)$ (which we will assume from now on) means that the map is unicellular ({\it i.e.} has only one face) and that the positions of the labels are determined by the choice of the edge labeled by $1$ (which can be seen as a \emph{root}). In this case, the couple of permutations is entirely determined by $\beta$.\\

Therefore, if $\lambda \vdash n$, the quantity $A(\lambda)$ is the number of
 rooted unicellular maps  with black vertices' degree distribution $\lambda$
 (there are no conditions on white vertices).
The condition that the product $(1\ 2\ \ldots\ n) \beta^{-1}$ is a long cycle is equivalent
to the fact that the corresponding rooted bipartite map has only one white vertex
(we call such maps \emph{star} maps).
Thus $B(\lambda)$ is the number of star
 rooted unicellular maps  with black vertices' degree distribution $\lambda$.\\
 
As in the papers \cite{SchaefferVassilieva:factorizations_long_cycle} and \cite{MoralesVassilieva:factorizations_long_cycle}, our combinatorial construction deals with maps with additional structure:

\begin{Def}
 A black-partitioned (rooted unicellular) map is a rooted unicellular map with a set partition $\pi$ of its black vertices. We call degree of a part (block) $\pi_{i}$ of $\, \pi$ the sum of the degrees of the vertices in $\pi_{i}$. The \emph{type} of a black-partitioned map is its blocks' degree distribution.
\end{Def}

In terms of permutations, a black-partitioned map consists of a couple $(\alpha,\beta)$ in $S_n$ with the condition $\alpha \, \beta=(1\ 2\ \ldots\ n)$ and a set partition $\pi$ of $\{1,\ldots,n\}$ coarser than the set partition in orbits under the action of $\beta$. 
Note that couples $(\alpha,\beta)$ with $\alpha \, \beta=(1\ 2\ \ldots\ n)$
are in bijection with permutations $\beta$.
Therefore, a black-partitioned map is the same object as a colored permutation
(see Definition \ref{DefColoredPerm}).
The number $p$ of colors corresponds to
the number of blocks in the set partition $\pi$.

\begin{example}
\label{exspm}
 Let $\beta = (1)(2 5)(3 7)(4)(6)$, $\alpha = (1 2 3 4 5 6 7) \beta^{-1} = (1 2 6 7 4 5 3)$,  and $\pi$ be the partition $\left\{ \{1, 3, 6, 7\};\{2, 5\};\{4\}\right\}$.
 Here, the type of $(\beta,\pi)$ is $(4,2,1)$.
Associating the triangle, circle and square shape to the blocks, $(\beta, \pi)$ is the black-partitioned star map pictured on figure \ref{spm}.
\begin{figure}[h]
\label{spm}
\begin{center}
\includegraphics[width=50mm]{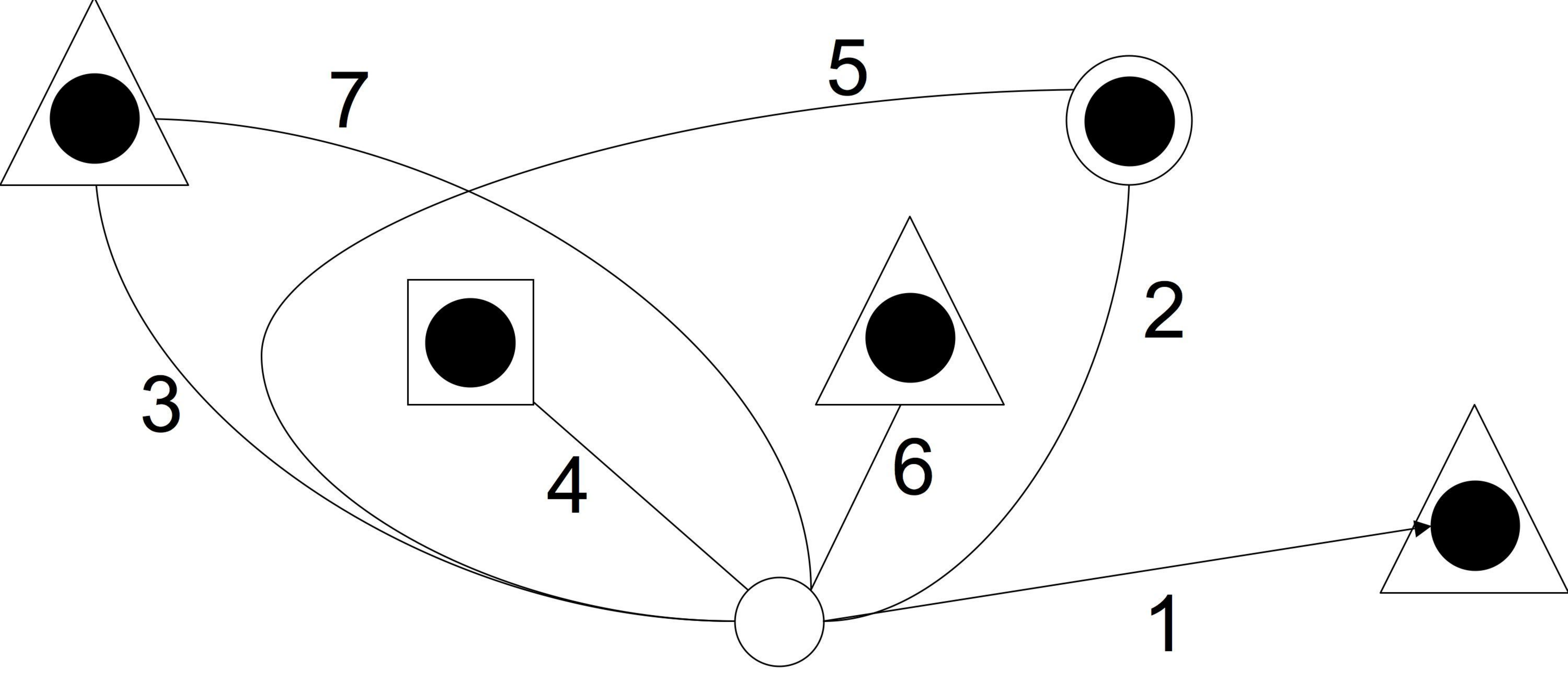}
\end{center}
\caption{The black-partitioned map defined in example \ref{exspm}} 
\end{figure}
\end{example}

If $\lambda \vdash n$,
we denote by $C(\lambda)$ (resp. $D(\lambda)$) the number of black-partitioned maps (resp. black-partitioned star maps) of type $\lambda$. Equivalently, $C(\lambda)$ (resp. $D(\lambda)$) is the number of couples $(\beta,\pi)$ as above such that $\pi$ is a partition of type $\lambda$ (resp. and $(1\ 2\ \ldots\ n) \beta^{-1}$ is a long cycle). 

With this notations, Theorem \ref{th:reformulationfine} can be rewritten as:
\begin{equation}\label{EqCD}
        D(\pnv\lambda) = \frac{1}{n-\ell(\lambda)+1} C(\pnv\lambda) \text{ ,  } \lambda \vdash n 
\end{equation}

\subsection{Permuted star thorn trees and Morales'-Vassilieva's bijection}\label{subsect:thorn_trees}

The main tool of this article is to encode black-partitioned maps into star thorn trees, which have a very simple combinatorial structure. Note that they are a particular case of the notion of thorn trees, introduced by A. Morales and the second author in \cite{MoralesVassilieva:factorizations_long_cycle}.

\begin{Def}[star thorn tree]
% We will consider \emph{(ordered rooted bipartite) star tree}, that is to say that the tree has exactly one white vertex, which is the root and $p$ ordered black vertices linked to it.\\
An \emph{(ordered rooted bipartite) star thorn tree} of size $n$ is a planar tree
with a white root vertex, $p$ black vertices and
$n - p$ thorns connected to the white vertex and $n - p$ thorns connected to the black vertices.
A thorn is an edge connected to only one vertex.
``Planar`` means that the sons of a given vertex are ordered
(here, a thorn should be considered as a son of its extremity).

%We call a thorn an edge connected to only one vertex. An \emph{(ordered rooted bipartite) star thorn tree} of size $n$ is a star tree with $n - p$ thorns connected to the white vertex and $n - p$ thorns connected to the black vertices, where $p$ is its number of black vertices.\\
We call \emph{type} of a star thorn tree its black vertices' degree distribution (taking the thorns into account).
%(of course, thorns must be counted to evaluate the degrees).
If $\mu$ is an integer partition, we denote by $\BT(\mu)$ the number of star thorn trees of type $\mu$.
\end{Def}

Two examples are given on Figure \ref{fig:ex_arbre_permute}
(for the moment, please do not pay attention to the labels).
The interest of this object lies in the following theorem.

\begin{Th}[\cite{MoralesVassilieva:factorizations_long_cycle}]\label{th:bij_MV}
Let $\mu \vdash n$ be a partition of length $p$. One has:
\begin{equation}\label{eq:bij_MV}
C(\mu) = (n-p)! \cdot \BT(\mu).
\end{equation}
\end{Th}
This theorem corresponds to the case $\lambda=(n)$ of
\cite[Theorem 2]{MoralesVassilieva:factorizations_long_cycle}
(note that the proof is entirely bijective).

The right-hand side of \eqref{eq:bij_MV} is the number of couples $(\tau,\sigma)$ where:
\begin{itemize}
 \item $\tau$ is a star thorn tree of type $\mu$.
 \item $\sigma$ is a bijection between thorns with a white extremity and thorns with a black extremity (by definition, $\tau$ has exactly $n-p$ thorns of white extremity and $n-p$ thorns of black extremity).
\end{itemize}
We call such a couple a permuted (star) thorn tree. By definition, the type of $(\tau,\sigma)$ is the type of $\tau$. Examples of graphical representations are given on Figure \ref{fig:ex_arbre_permute}: we put symbols on edges and thorns with the following rule.
Two thorns get the same symbol if they are associated by $\sigma$ and,
except from that rule, all symbols are different (the chosen symbols and their order do not matter, we call that a symbolic labeling).\\

\begin{figure}[ht]
$$\includegraphics[width=30mm]{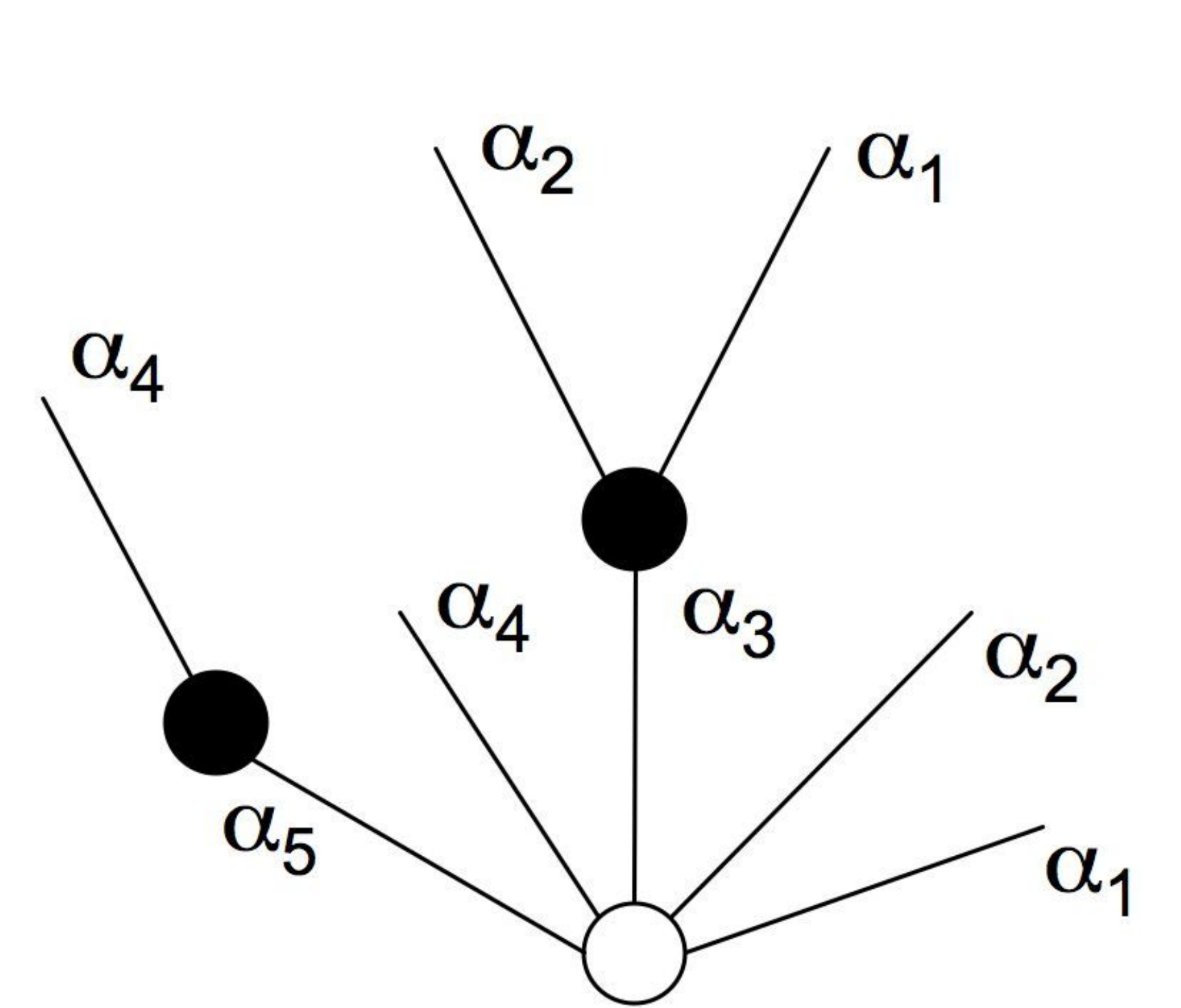} \hspace{2cm}
\includegraphics[width=3cm]{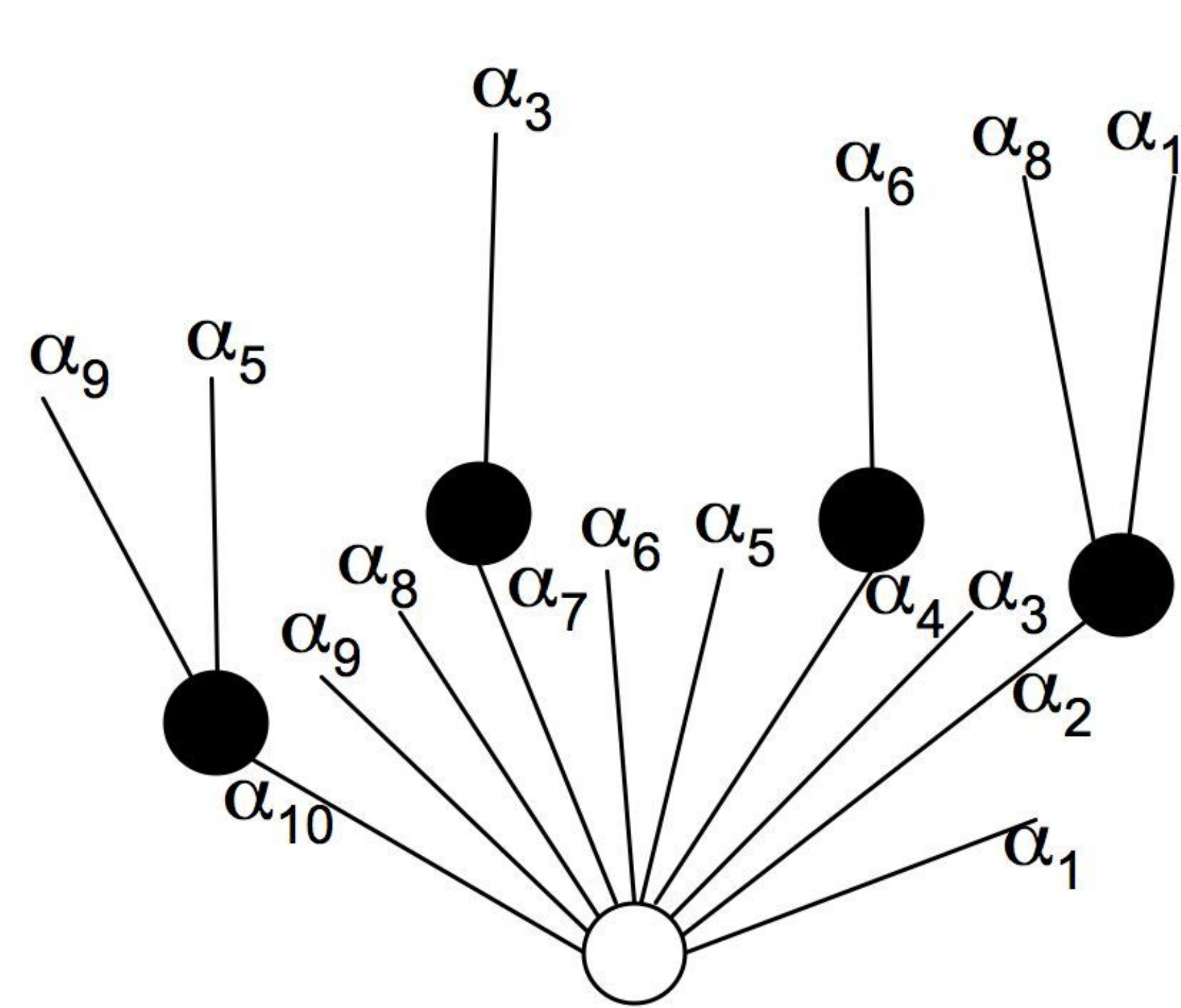}$$
\caption{Example of two permuted star thorn trees $\ex{1}$ of type $2^1 3^1$ and $\ex{2}$ of type $2^2 3^2$}
\label{fig:ex_arbre_permute}
\end{figure}

Using this result, one obtains another equivalent formulation for Theorem \ref{th:reformulationfine}:
\begin{equation}\label{EqDBT}  
    D(\pnv\lambda) = \frac{1}{n-p+1} (n-p)! \BT(\pnv\lambda)  \text{ , } \lambda \vdash n.
\end{equation}

Sections \ref{sect:def_psi}, \ref{sect:inj_psi} and \ref{sect:im_psi} are devoted to the proof of equation \eqref{EqDBT}. We proceed in a three step fashion. Firstly, we define a mapping $\Psi$ from the set of black-partitioned star maps of type $\lambda$ (counted by $D(\lambda)$) into the set of permuted star thorn trees of the same type. Secondly, we show it is injective. As a final step, we compute the cardinality of the image set of $\Psi$ and show it is exactly $\left (1/(n-p+1)\right)(n-p)! \BT(\pnv\lambda)$.\\
%Thanks to the exact knowledge of the number of white vertices in $D(\lambda)$ (one), we are able to$\Psi$ as a new mapping (different from the one in \cite{MoralesVassilieva:factorizations_long_cycle}), simpler to define  and giving more insight on the objects.

\noindent {\em Remark.} Although there are some related ideas, $\Psi$ is not the restriction of the bijection of paper \cite{MoralesVassilieva:factorizations_long_cycle}. 

\section{Mapping black-partitioned star maps to permuted thorn trees}\label{sect:def_psi}

\subsection{Labeled thorn tree}
Let $(\beta,\pi)$ be a black-partitioned star map. First we construct a labeled star thorn tree $\overline{\tau}$:
\begin{enumerate}[(i)]

\item \label{item:root_place} Let $(\alpha_k)_{(1\leq k \leq n)}$ be the integer list such that $\alpha_1 = 1$ and such that the long cycle $\alpha=(1\ 2\ \ldots\ n) \beta^{-1}$ is equal to $(\alpha_1 \alpha_2 \alpha_3 \dots \alpha_n)$. The root of $\overline{\tau}$ is a white vertex with $n$ descending edges labeled from right to left with $\alpha_1,  \alpha_2,  \alpha_3,  \dots, \alpha_n$ ($\alpha_1$ is the rightmost descending edge and $\alpha_n$ the leftmost). 

\item \label{item:edge_root} Let $m_i$ be the maximum element of the block $\pi_i$. For $k=1\ldots n$, if $\alpha_k = \beta(m_i)$ for some $i$, we draw a black vertex at the other end of the descending edge labeled with $\alpha_k$. Otherwise the descending edge is a thorn.
\begin{Rem}\label{rem:leftmost_edge}
As $\alpha_n=\alpha^{-1}(1) = \beta(n)$ the leftmost descending edge is never a thorn
and is labeled with $\beta(n)$.
\end{Rem}

\item \label{item:black_thorns} For $i \in \{1,\dots,p\}$, let ${(\beta^{u}_{1} \ldots \beta^{u}_{l_u})}_{1\leq u \leq c}$ be the $c$ cycles included in block $\pi_i$ such that $\beta^{u}_{l_u}$ is the maximum element of cycle $u$. (We have $\Sigma_{u} l_u = \mid\pi_i\mid$). We also order these cycles according to their maximum, \textit{i.e.} we assume that $ \beta^{c}_{l_c} < \beta^{c-1}_{l_{c-1}} < \ldots < \beta^{1}_{l_1} = m_i$. As a direct consequence, $\beta^{1}_1 = \beta(m_i)$.\\
We connect $\mid\pi_i\mid-1$ thorns to the black vertex linked to the root by the edge $\beta(m_i)$. Moving around the vertex clockwise and starting right after edge $\beta(m_i)$, we label its thorns with the integers 
\[\beta^{c}_{l_c}, \dots, \beta^{c}_{1}, \dots, \beta^{2}_{l_2}, \dots, \beta^{2}_{1}, \beta^{1}_{l_1},  \ldots, \beta^{1}_{2}\]
in this order. Note that the last one is $\beta^{1}_{2}$ as $\beta^{1}_1 = \beta(m_i)$ is the label of the edge. Then $\overline{\tau}$ is the resulting thorn tree.

\begin{Rem}\label{rem:betai}
Moving around a black vertex clockwise starting with the thorn right after the edge, a new cycle of $\beta$ begins whenever we meet a left-to-right maximum of the labels. 
\end{Rem}

\end{enumerate}

The idea behind this construction is to add a root to the map $(\alpha,\beta)$, select one edge per block, cut all other edges into two thorns and merge the vertices corresponding to the same black block together. Step (\ref{item:root_place}) tells us where to place the root, step (\ref{item:edge_root}) which edges we select and step (\ref{item:black_thorns}) how to merge vertices (in maps unlike in graphs, one has several ways to merge given vertices).

\begin{example}
\label{exspm2lstt}
Let us take the black-partitioned star map of example \ref{exspm}. Following construction rules (\ref{item:root_place}) and (\ref{item:edge_root}), one has $m_\triangle = 7$, $m_\bigcirc = 5$, $m_\Box = 4$ and the descending edges indexed by $\beta(m_\triangle) = 3$, $\beta(m_\bigcirc) = 2$ and $\beta(m_\Box) = 4$ connect a black vertex to the white root. Other descending edges from the root are thorns. Using (\ref{item:black_thorns}), we add labeled thorns to the black vertices to get the labeled thorn tree depicted on Figure \ref{taubar}. Focusing on the one connected to the root through the edge $3$, we have $(\beta^1_1 \beta^1_{2}) (\beta^2_{1}) (\beta^3_{1}) = (3 7)(6)(1)$. Reading the labels clockwise around this vertex, we get $1,6,7,3$. The three cycles can be simply recovered looking at the left-to-right maxima $1$, $6$ and $7$.

\begin{figure}[ht]
\begin{center}
\includegraphics[width=120mm]{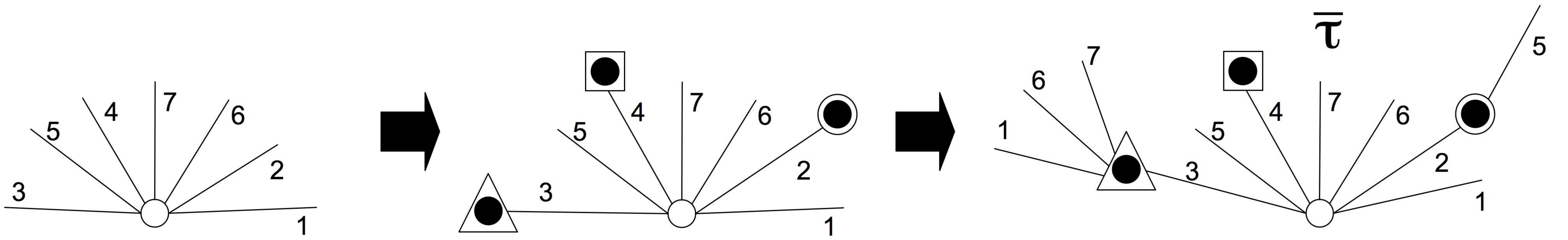} %\hspace{3cm} \includegraphics[width=100mm]{Tau}
\end{center}
\caption{Labeled thorn tree associated to the black-partitioned star map of Figure \ref{spm}} \label{taubar}
\end{figure}
\end{example}

\begin{Rem}\label{rem:taubar2smp}
Let us fix a labeled thorn tree $\overline{\tau}$ coming from a black-partitioned star map $(\beta,\pi)$.
Then $\alpha = (1\ 2\ \dots\ n) \beta^{-1}$ can be found from $\overline{\tau}$ by reading the labels around the root in counter-clockwise order and $\pi$ is the following set-partition: for each black vertex $b$ of  $\overline{\tau}$, the block $\pi_b$ of $\pi$ is the set of the labels of the edge and of the thorns linked to $b$.
Hence, a labeled thorn tree $\overline{\tau}$ corresponds at most to one black-partitioned star map $(\beta,\pi)$.
\end{Rem}

\subsection{Permuted thorn tree}

We call $\tau$ the star thorn tree obtained from $\overline{\tau}$ by removing labels and $\sigma$ the permutation that associates to a white thorn in $\tau$ the black thorn with the same label in $\overline{\tau}$.

Finally, we define:
$
\Psi(\beta,\pi) = (\tau,\sigma).
$
\begin{example}
Following up with example \ref{spm}, we get the permuted thorn tree $\ex{3}$ drawn on Figure \ref{tausigma}. Graphically we use the same convention as in paragraph \ref{subsect:thorn_trees} to represent $\sigma$.
 \begin{figure}[h]
 \begin{center}
 \includegraphics[width=100mm]{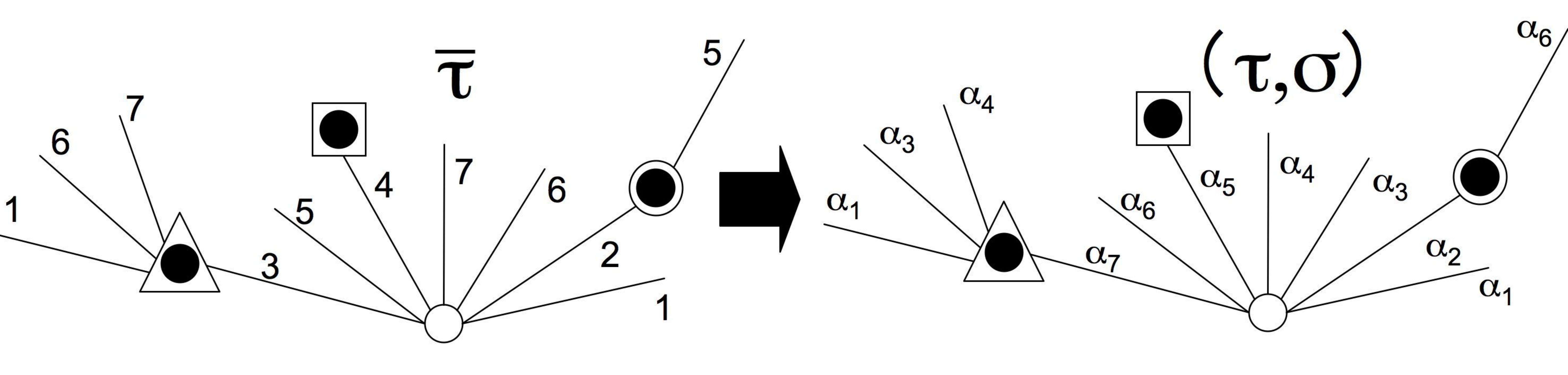}
 \end{center}
 \caption{Permuted thorn tree $\ex{3}$ associated to the black-partitioned star map of Figure \ref{spm}} \label{tausigma}
 \end{figure}
% \noindent Then, we have: $(a_1,a_2,a_3,a_4) = (1,7,6,5)$ and $(b_1,b_2,b_3,b_4) = (5,7,6,1)$. As a result,
% \begin{equation}
% \sigma = \left (\begin{array}{cccclcccc}1&2&3&4\\ 4&2&3&1\end{array} \right )
% \end{equation}
\end{example}

%%%%%%%%%%%%%%%%%%%%%%%%%%%%%%%%%%%%%%%%%%%%%%%%
%%%%%%%%%      new section     %%%%%%%%%%%%%%%%%
%%%%%%%%%%%%%%%%%%%%%%%%%%%%%%%%%%%%%%%%%%%%%%%%

\section{Injectivity and reverse mapping}\label{sect:inj_psi}
Assume $(\tau,\sigma) = \Psi(\beta,\pi)$ for some black partitioned star map $(\beta,\pi)$. We show that $(\beta,\pi)$ is actually uniquely determined by $(\tau,\sigma)$.\\
As a first step, we recover the labeled thorn tree $\overline{\tau}$. Let us draw the permuted thorn tree $(\tau,\sigma)$ as explained in paragraph \ref{subsect:thorn_trees}. We show by induction that there is at most one possible integer value for each symbolic label.

\begin{enumerate}[(i)]

\item By construction, the label $\alpha_1$ of the right-most edge or thorn descending from the root is necessarily $1$.
%If $\alpha_1$ is symbolic label of a thorn, then we assess the (value) label $1$ to the corresponding black thorn.

\item \label{item:etape_determiner_bi} Assume that for $i \in [n-1]$, we have identified the symbols of values $1,2,\ldots,i$. We look at the edge or thorn with label $i$ connected to a black vertex $b$. In this step, we determine which symbol corresponds to $\beta(i)$.\\
    Recall that, when we move around $b$ clockwise finishing with the edge (in this step, we will always turn in this sense), a new cycle begins whenever we meet a left-to-right maximum (Remark \ref{rem:betai}). So, to find $\beta(i)$, one has to know whether $i$ is a left-to-right maximum or not.\\
If all values of symbols of thorns before $i$ have not already been retrieved, then $i$ is not a left-to-right maximum. Indeed, the remaining label values are $i+1$, \ldots, $n$ and at least one thorn's label on the left of $i$ lies in this interval. According to our construction  $\beta(i)$ necessarily corresponds to the symbolic label of the thorn right at the left of $i$ (case {\it a})\\
If all the symbol values of thorns before $i$ have already been retrieved (or there are no thorns at all), then $i$ is a left-to-right maximum. According to the construction of $\overline{\tau}$, $\beta(i)$ corresponds necessarily to the symbolic label of the thorn preceding the next left-to-right maximum% (moving around the black vertex from left-to-right, i.e. $\bf clockwise$)
. But one can determine which thorn (or edge) corresponds to the next left-to-right maximum: it is the first thorn (or edge) $e$ whose value has not been retrieved so far (again moving around the black vertex from left to right). Indeed, all the values retrieved so far are less than $i$ and those not retrieved greater than $i$. Therefore $\beta(i)$ is the thorn right at the left of $e$ (case {\it b}).

If all the values of the labels of the thorns connected to $b$ have already been retrieved then $i$ is the maximum element of the corresponding block and $\beta(i)$ corresponds to the symbolic label of the edge connecting this black vertex to the root (we can see this as a special case of case {\it b}).

\item \label{item:etape_determiner_ipu} Consider the element (thorn of edge) of white extremity with the symbolic label corresponding to $\beta(i)$. The next element (turning around the root in counter-clockwise order) has necessarily label $\alpha(\beta(i))=i+1$.\\
As a result, the knowledge of the thorn or edge with label $i$ uniquely determines the edge or thorn with label $i+1$.
\end{enumerate}

Applying the previous procedure up to $i=n-1$ we see that $\overline{\tau}$ is uniquely determined by $(\tau,\sigma)$ and so is $(\beta, \pi)$ (see Remark \ref{rem:taubar2smp}).

\begin{example}
Take as an example the permuted thorn tree $\ex{1}$ drawn on the left-hand side of Figure \ref{fig:ex_arbre_permute}, the procedure goes as described on Figure \ref{recons}. First, we identify $\alpha_1 = 1$. Then, as there is a non (value) labeled thorn $\alpha_2$ on the left of the thorn connected to a black vertex with label value $1$, necessarily $1$ is not a left-to-right maximum and $\alpha_2$ is the label of the thorn immediately to the left of $1$. Then as $\alpha_3$ follows $\alpha_2 = \beta(1)$ around the white root, we have $\alpha_3 = \alpha(\beta(1))=2$.

% \begin{figure}[h]
% \begin{center}
% \includegraphics[width=30mm]{taubarsymbols1_1}\hspace{15mm}
% %\includegraphics[width=40mm]{taubarsymbols3}\hfill
% \includegraphics[width=30mm]{taubarsymbols5}\hspace{15mm}
% \includegraphics[width=30mm]{taubarsymbols7}
% \caption{First steps of the reconstruction}\label{recons}
% \end{center}
% \end{figure}

\noindent We apply the procedure up to the full retrieval of the edges' and thorns' labels. We find $\alpha_2 = 3$, $\alpha_4 = 4$, $\alpha_5 = 5$. Finally, we have $\alpha = (1 3 2 4 5)$, $\beta = (213)(4)(5)$, $\pi = \{\{1, 2, 3\}; \{4, 5\} \}$ as shown on figure \ref{endrecons}.

\begin{figure}[ht]
\begin{center}
\includegraphics[width=30mm]{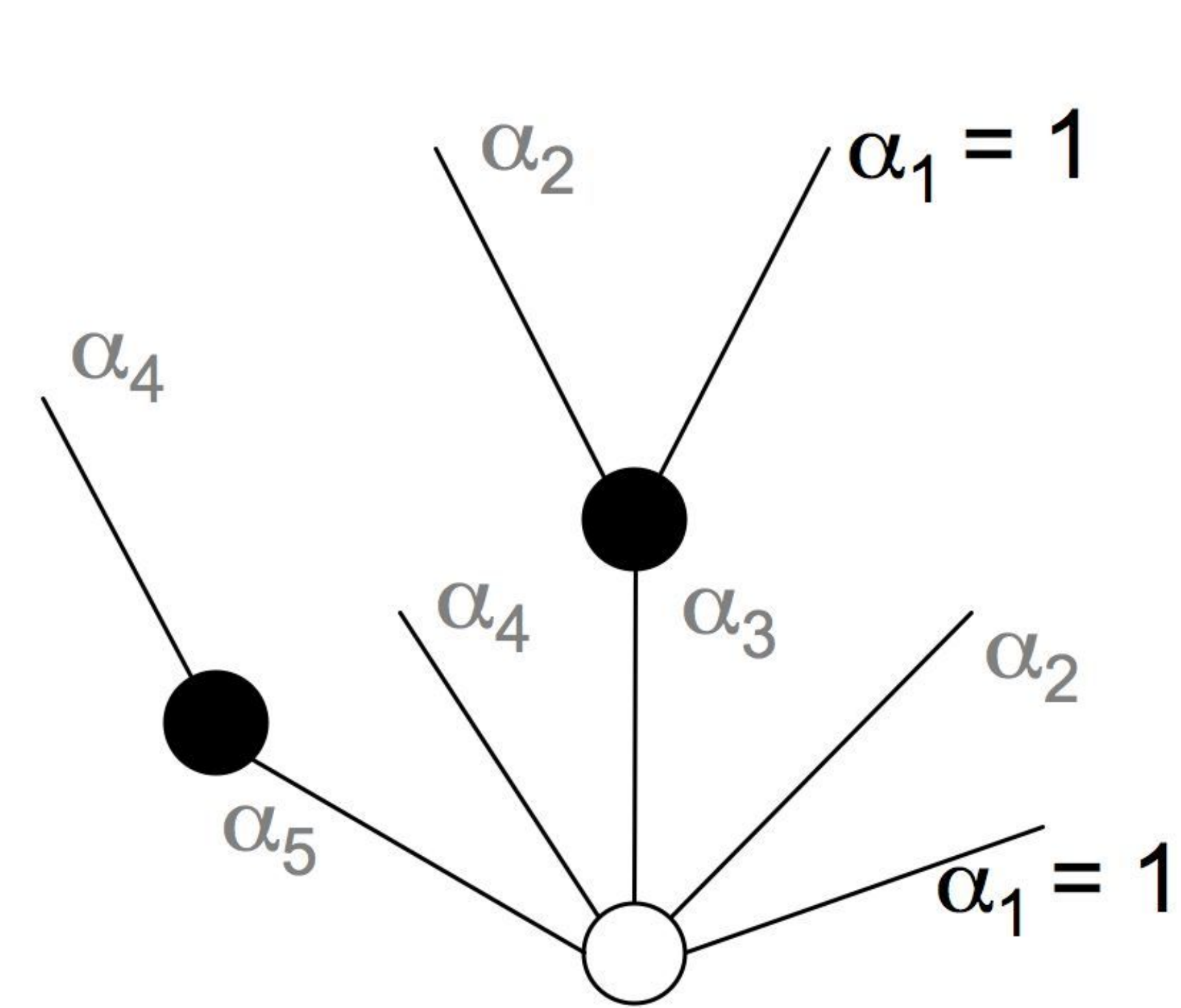}\hspace{10mm}
\includegraphics[width=30mm]{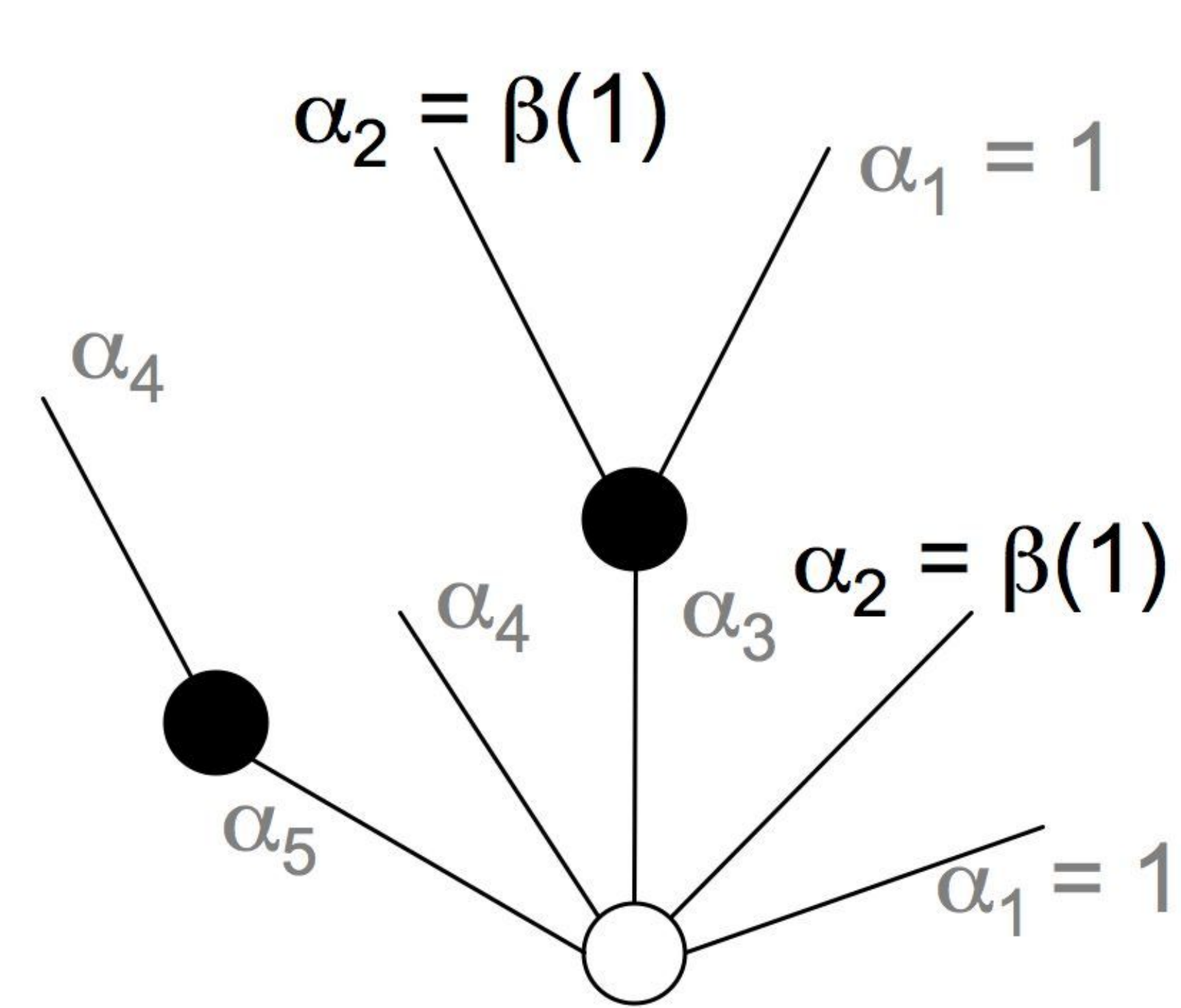}\hspace{10mm}
\includegraphics[width=30mm]{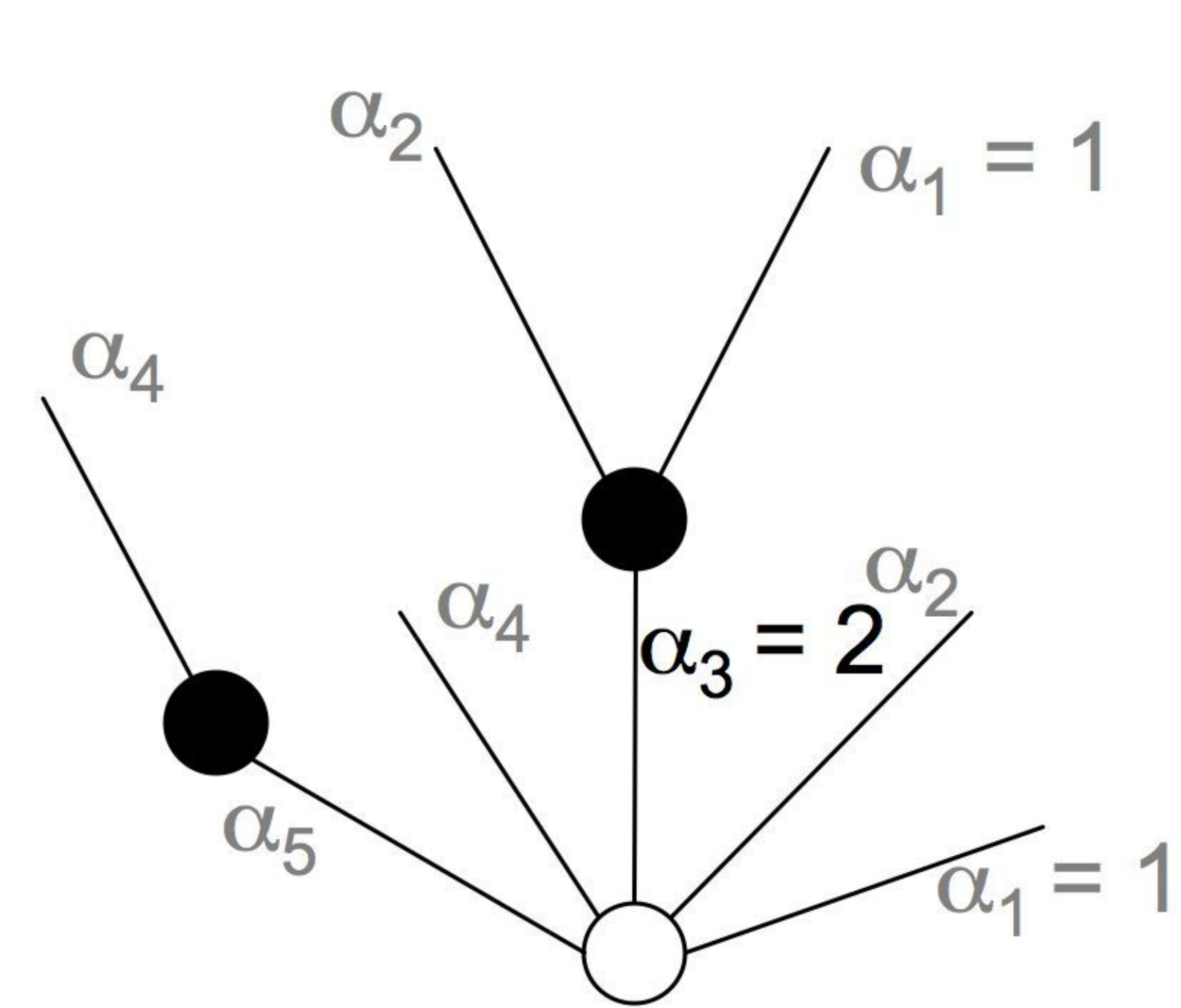}
\end{center}
\begin{center}
\includegraphics[width=30mm]{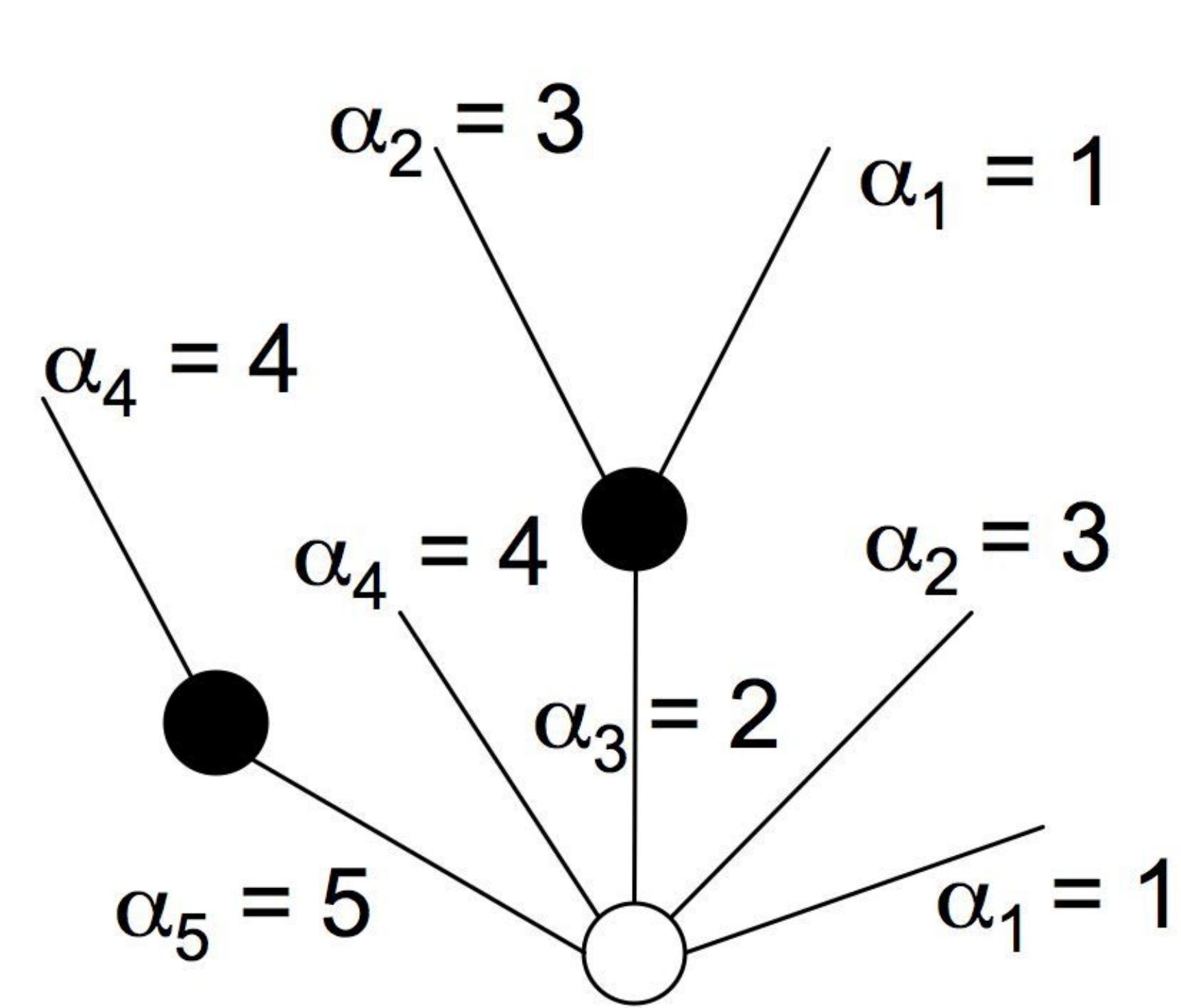}\
\hspace{15mm}
\includegraphics[width=30mm]{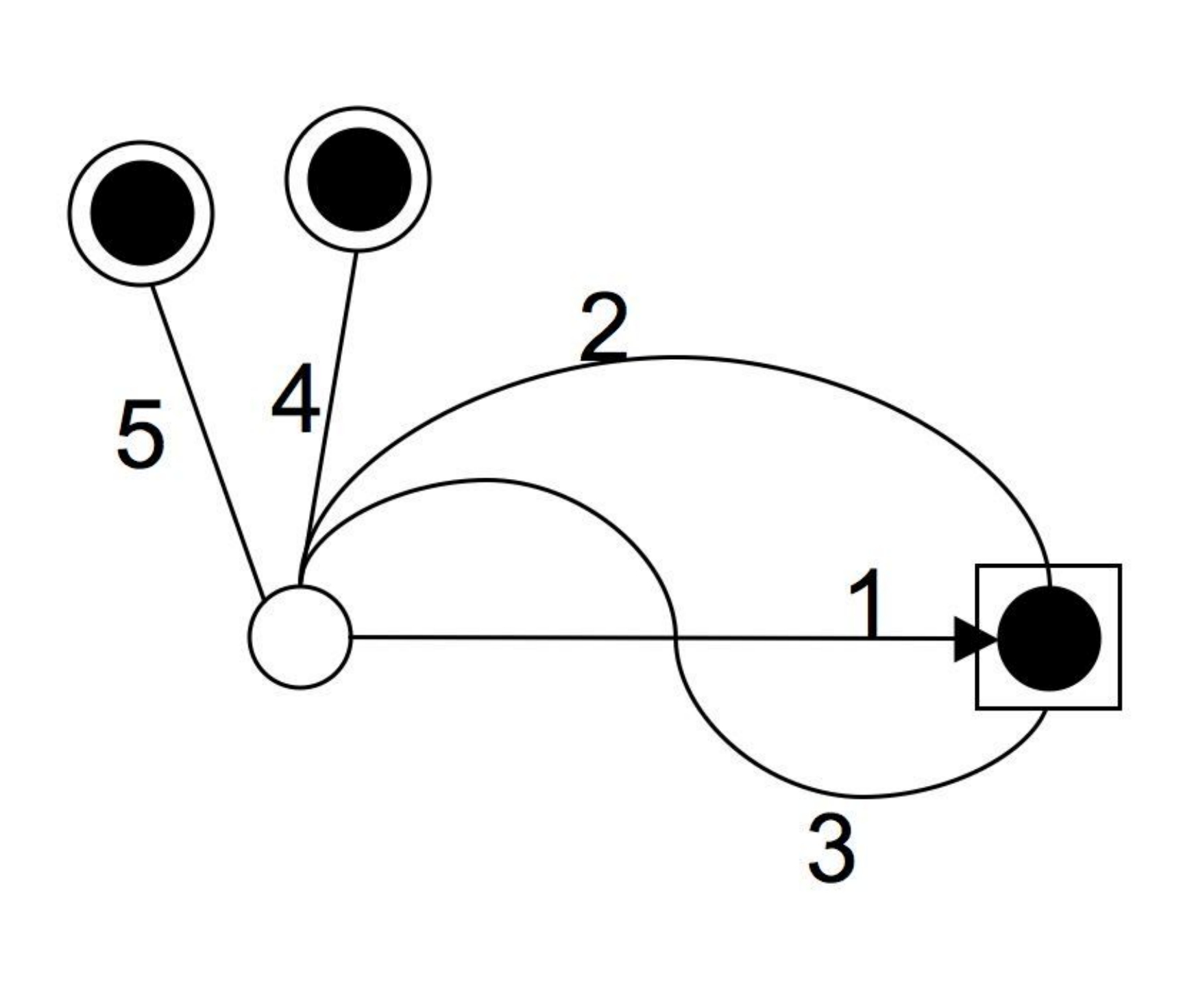}
\caption{Reconstruction of the map}\label{recons}\label{endrecons}
\end{center}
\end{figure}
\end{example}

%%%%%%%%%%%%%%%%%%%%%%%%%%%%%%%%%%%%%%%%%%%%%%%%
%%%%%%%%%      new section     %%%%%%%%%%%%%%%%%
%%%%%%%%%%%%%%%%%%%%%%%%%%%%%%%%%%%%%%%%%%%%%%%%

\section{Characterisation and size of the image set $\Im(\mm)$}\label{sect:im_psi}

\subsection{A necessary and sufficient condition to belong to $\Im(\mm)$}

\subsubsection{Why $\mm$ is not surjective?}
%The goal of this paragraph is to characterize the permuted thorn trees $(\tau,\sigma)$ which are in the image $\Im(\mm)$ of $\mm$.
Let us fix a permuted star thorn tree $(\tau,\sigma)$. We can try to apply to it the procedure of section \ref{sect:inj_psi} and we distinguish two cases:
\begin{itemize}
\item it can happen, for some $i<n$, when one wants to give the label $i+1$ to the edge following $\beta(i)$ (step (\ref{item:etape_determiner_ipu})), that this edge has already a label $j$ ($j<i$). If so, the procedure fails and $(\tau,\sigma)$ is not in $\Im(\mm)$.
\item if this never happens, the procedure ends with a labeled thorn tree $\overline{\tau}$. In this case, one can find the unique black-partitioned star map $M$ corresponding to $\overline{\tau}$ and by construction $\mm(M)=(\tau,\sigma)$.
\end{itemize}
For instance, take the couple $\ex{2}$ on the right of Figure \ref{fig:ex_arbre_permute}, the procedure gives successively 
\[ \alpha_1=1,\ \alpha_9=2,\ \alpha_{10}=3,\ \alpha_6=4,\ \alpha_7=5,\ \alpha_4=6,\ \alpha_5=7\]
and then we should choose $\alpha_1=8$, but this is impossible because we already have $\alpha_1=1$.
\begin{Lem}
 If the procedure fails, the label $j$ of the edge that should get a second label $i+1$ is always $1$.
\end{Lem}
\begin{proof}
Assume $j >1$. As the reconstruction procedure did not fail for $1\ldots i$, there are two distinct pairs of thorns with labels $i$ and $j-1$. We will prove that the reconstruction provides labels $\beta(i)$ and $\beta(j-1)$ to two distinct elements.

We assume that the labels $\beta(i)$ and $\beta(j-1)$ have been given to the same element.
In particular, $i$ and $j-1$ must belong to the same black vertex.
Let us consider the different possible cases in the reconstruction step (\ref{item:etape_determiner_bi}):
\begin{itemize}
\item If $\beta(j-1)$ is obtained {\it via} case {\it b} (the left-to-right maximum case), the label $i$ must be just to the right of $\beta(j-1)$ and not a left-to-right maximum. But this is impossible because all thorns to the left of $\beta(j-1)$ (including $\beta(j-1)$) have labels smaller than $j$.
\item If  $j-1$ is obtained {\it via} case {\it a} (the not left-to-right maximum case) and $i$ is a left-to-right maximum. The label $j-1$ is just to the right of the thorn/edge labeled by both $\beta(j-1)$ and $\beta(i)$. Then $\beta(i)$ is before the next left-to-right maximum. So the edge to the right of $\beta(i)$ has a label greater than $i$ and can not be $j-1$.
\item If  $j-1$ is obtained {\it via} case {\it a} (the not left-to-right maximum case) and $i$ is {\em not} a left-to-right maximum. The label $j-1$ is still just to the right of the thorn/edge labeled by both $\beta(j-1)$ and $\beta(i)$. Label $i$ must be as well just to the right of $\beta(i)$. It is not possible as $i$ and $j-1$ are the labels of two distinct thorns or edge since the procedure has not failed at step $i$.

\end{itemize}
Finally $\beta(i)$ and $\beta(j-1)$ correspond to two different symbolic labels and hence $i+1$ and $j$ also (they are respectively the symbolic label of the elements right at the left of $\beta(i)$ and $\beta(j-1)$ when turning around the root). Hence, the procedure can not fail for a value of $j >1$.
\end{proof}

\subsubsection{An auxiliary oriented graph}
Remark \ref{rem:leftmost_edge} gives a necessary condition for $(\tau,\sigma)$ to be in $\Im(\mm)$: its leftmost edge attached to the root must be a real edge and not a thorn.
From now on, we call this property $(P1)$: note that, among all permuted thorn trees of a given type $\lambda \vdash n$ of length $p$, exactly $p$ over $n$ have this property.
Whenever $(P1)$ is satisfied, we denote $e_0$ the left-most edge leaving the root and $\pi_0$ its black extremity.
The lemma above shows that the procedure fails if and only if $e_0$ is chosen as $\beta(i)$ for some $i<n$. But this can not happen at any time. Indeed, the following lemma is a direct consequence from step (\ref{item:etape_determiner_bi}) of the reconstruction procedure:
\begin{Lem}\label{lem:sommet_complete}
    A real edge (\textit{i.e.} which is not a thorn) $e$ can be chosen as $\beta(i)$ only if the edge and all thorns attached to the corresponding black vertex have labels smaller or equal to $i$. If this happens, we say that the black vertex is {\em completed} at step $i$.
\end{Lem}
%This lemma has an immediate corollary:
\begin{Corol}
 Let $e$ be a real edge of black extremity $\pi \neq \pi_0$. Let us denote $e'$ the element (edge or thorn) immediately to the left of $e$ around the white vertex. Let $\pi'$ be the black extremity of the element $e''$ associated to $e'$ (\textit{i.e.} $e'$ itself if it is an edgeand its image by $\sigma$ otherwise). Then $\pi'$ can not be completed before $\pi$.
\end{Corol}
\begin{proof}
If $\pi'$ is completed at step $i$, by Lemma \ref{lem:sommet_complete},
the element $e''$ has a label $j \leq i$.
As $e'$ has the same label, this implies that $e$ has label $\beta(j-1)$
or in other words, that $\pi$ is completed at time $j-1 < i$.
\end{proof}

When applied for every black vertex $\pi \neq \pi_0$, this corollary gives some partial information on the order in which the black vertices can be completed. We will summarize this in an oriented graph $G(\tau,\sigma)$: its vertices are the black vertices of $\tau$ and its edges are $\pi \to \pi'$, where $\pi$ and $\pi'$ are in the situation of the corollary above.
This graph has one edge attached to each of its vertices except $\pi_0$. As examples, we draw the graphs corresponding to $\ex{2}$ and to $\ex{3}$ (see Figures \ref{fig:ex_arbre_permute} and \ref{tausigma}) on Figure \ref{fig:ex_graphe_ordre_completion}.

\begin{figure}[t]
$$\includegraphics[width=2cm]{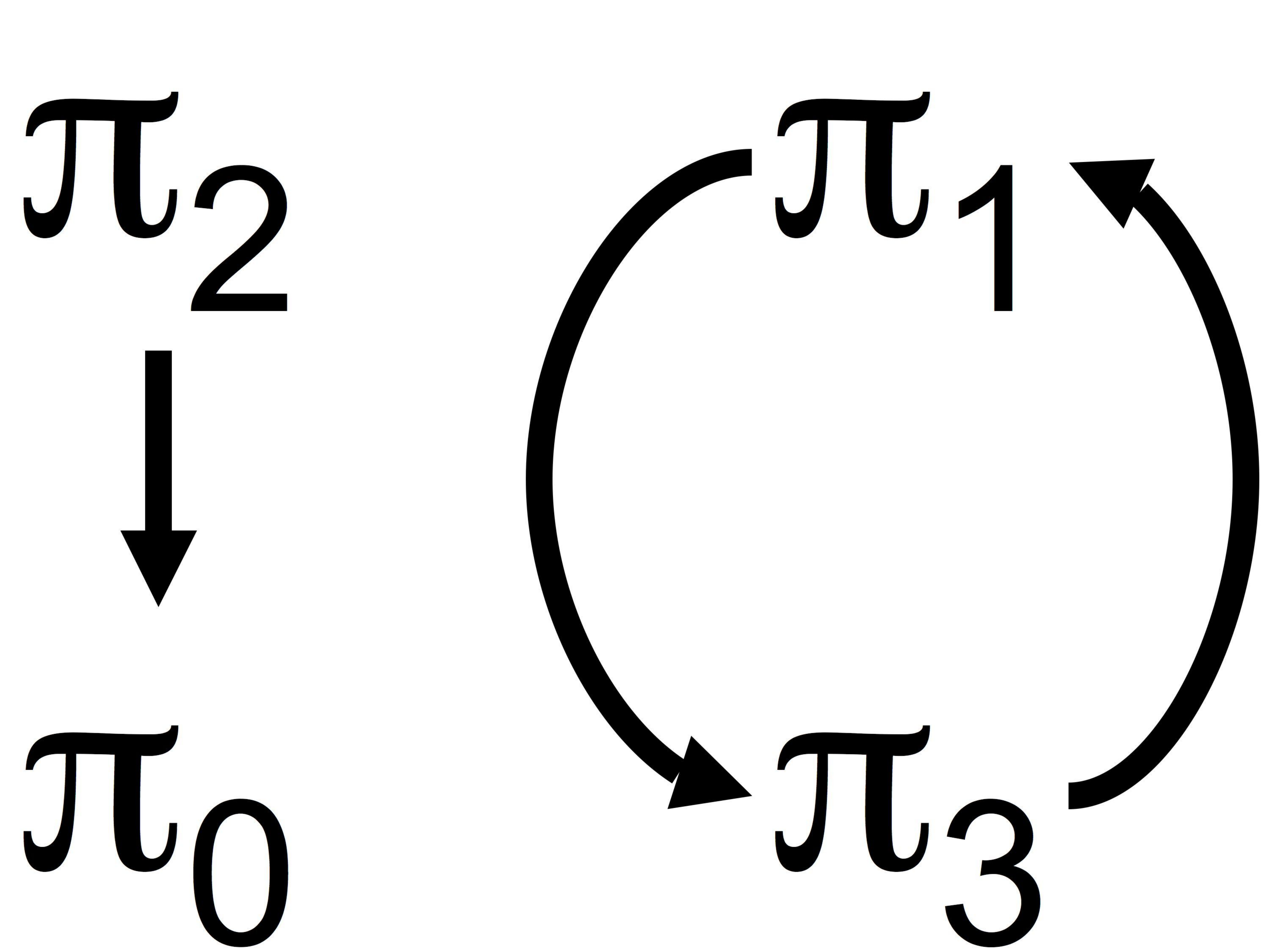} \qquad \qquad \includegraphics[width=2cm]{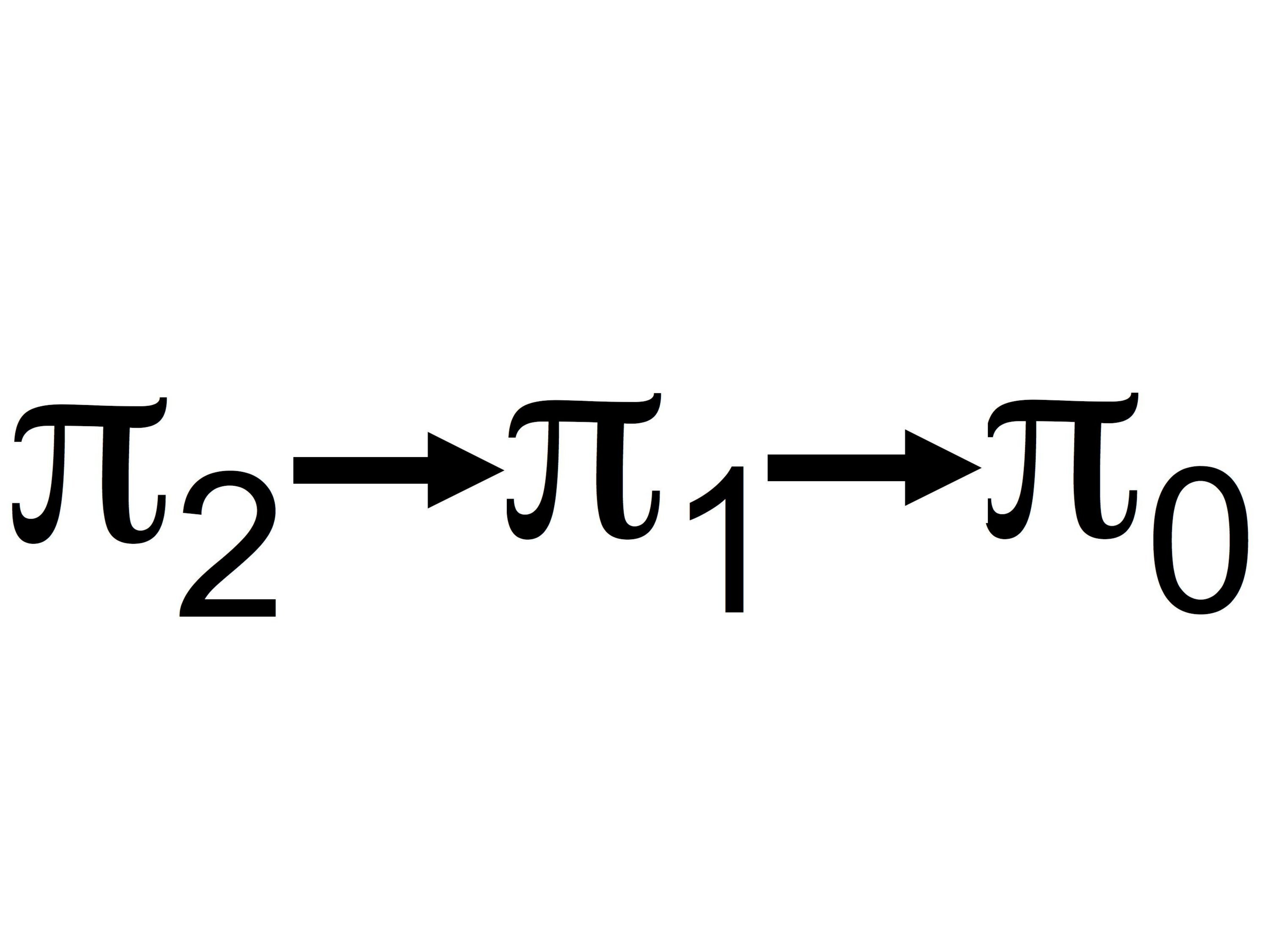}$$
\caption{Two examples of auxiliary graphs.}
\label{fig:ex_graphe_ordre_completion}
\end{figure}

\subsubsection{The graph $G(\tau,\sigma)$ gives all the information we need!}
Can we decide, using only $G(\tau,\sigma)$, whether $(\tau,\sigma)$ belongs to $\Im(\mm)$ or not? There are two cases, in which the answer is obviously yes:
\begin{enumerate}
 \item Let us suppose that $G(\tau,\sigma)$ is an oriented tree of root $\pi_0$ (all edges are oriented towards the root). In this case, we say that $(\tau,\sigma)$ has property $(P2)$. Then, the vertex $\pi_0$ can be completed only when all other vertices have been completed, \textit{i.e.} when all edges and thorns have already a label. That means that $e_0$ can be chosen as $\beta(i)$ only for $i=n$. Therefore, in this case, the procedure always succeeds and $(\tau,\sigma)$ belongs to $\Im(\mm)$. This is the case of $\ex{3}$.
 \item Let us suppose that $G(\tau,\sigma)$ contains an oriented cycle (eventually a loop). Then all the vertices of this cycle can never be completed. Therefore in this situation the procedure always fails and $(\tau,\sigma)$ does not belong to $\Im(\mm)$. This is the case of $\ex{2}$.
\end{enumerate}
In fact, we are always in one of these two cases:
\begin{Lem}
 Let $G$ be an oriented graph whose vertices have out-degree $1$, except one vertex $v_0$ which has out-degree $0$. Then $G$ is either an oriented tree with root $v_0$ or contains an oriented cycle.
\end{Lem}
\begin{proof}
    We consider two different cases:
    \begin{itemize}
        \item either, there exists a vertex $v$ with no paths from $v$ to $v_0$.
            In this case, we denote $v^1,v^2,\ldots$ the vertices such that
           $v^1$ is the successor of $v$ and $v^{i+1}$ is the successor of $v^i$.
            As the number of vertices is finite,
            there are at least two indices $i_1$ and $i_2$
            such that $v^{i_1}=v^{i_2}$.
            The chain $v^{i_1}v^{i_1+1}\ldots v^{i_2}$ is an oriented loop.
        \item or there is a path from each vertex $v$ to $v_0$.
            So $G$ contains an oriented tree of root $v_0$.
            As the number of edges is exactly one less than the number of vertices,
            $G$ is an oriented tree.\qedhere
    \end{itemize}
\end{proof}

Finally, one has the following result:
\begin{Prop}
 The mapping $\mm$ defines a bijection:
\begin{equation}\label{eq:main_bijection}
\left\{\begin{array}{c} \text{black-partitioned star maps} \\
          \text{of type }\lambda
         \end{array}\right\}
\simeq
\left\{\begin{array}{c} \text{permuted star thorn trees of type }\lambda\\
        \text{with properties }(P1) \text{ and }(P2)
       \end{array}\right\}.
\end{equation}
\end{Prop}

\subsection{Proportion of permuted thorn trees $(\tau,\sigma)$ in $\Im(\mm)$}

To finish the proof of equation \eqref{EqDBT},
one has justto compute the size of the right-hand side of \eqref{eq:main_bijection}.
We do it {\it via} a quite technical (but pretty easy) induction,
it would be nice to find a more elegant argument.

\begin{Prop}\label{prop:proportion_P2}
 Let $\lambda$ be a partition of $n$ of length $p$. Denote by $P(\lambda)$ %(resp. $P'(\lambda)$) 
 the proportion of couples $(\tau,\sigma)$ with properties $(P1)$ and $(P2)$ among all the permuted thorn trees of type $\lambda$ %(resp. among permuted thorn trees of type $\lambda$ with property $(P1)$)
. Then, one has:
$$P(\lambda)=\frac{1}{n-p+1}.$$
\end{Prop}

\begin{proof}
 In fact, we will rather work with the proportion $P'(\lambda)$ of couples verifying $(P2)$ among the permuted thorn trees of type $\lambda$ verifying $(P1)$. As the proportion of couples with property $(P1)$ among couples $(\tau,\sigma)$ of type $\lambda$ is $\ell(\lambda)/|\lambda|$, one has: $P'(\lambda)=|\lambda|/\ell(\lambda)\cdot P(\lambda)$. We will prove by induction over $p=\ell(\lambda)$ that:
$$P'(\lambda)= \frac{|\lambda|}{\ell(\lambda) (|\lambda| - \ell(\lambda)+1)}.$$

The case $p=1$ is easy: as $G(\tau,\sigma)$ has only one vertex and no edges, it is always a tree.
Therefore, for any $n \geq 1$, one has $P'((n))=1$.\bigskip

Suppose that the result is true for any $\lambda$ of length $p-1$ and fix a partition $\mu \vdash n$ of length $p>1$. 

Let $PTT_1(\mu)$ (resp. $PTT_{1,2}(\mu)$) be the set of permuted
thorn trees $(\tau,\sigma)$ of type $\mu$, verifying $(P1)$
(resp. verifying $(P1)$ and $(P2)$).
With these notations, $P'(\mu)$ is defined as the quotient
\[ \frac{\big|PTT_{1,2}(\mu)\big|}{\big|PTT_{1}(\mu)\big|}. \]
It will be convenient to consider marked permuted thorn trees, {\it i.e.} permuted thorn trees
with a marked black vertex different from $\pi_0$.
The marked vertex will be denoted $\overline{\pi}$ and the corresponding edge
$e_{\overline{\pi}}$.
We denote $MPTT_1(\mu)$ (resp. $MPTT_{1,2}(\mu)$) the set of marked permuted
thorn trees $(\tau,\sigma)$ of type $\mu$, verifying $(P1)$
(resp. verifying $(P1)$ and $(P2)$).
To each permuted thorn tree $(\tau,\sigma)$ of type $\mu$ corresponds exactly 
$p-1$ marked permuted thorn trees, so:
\begin{align*}
\big|MPTT_\star(\mu)\big| &= (p-1) \cdot \big|PTT_\star(\mu)\big|
\text{ for $\star=1$ or $\star=1,2$},\\
\text{and thus }P'(\mu) &= \frac{\big|MPTT_{1,2}(\mu)\big|}{\big|MPTT_{1}(\mu)\big|}. 
\end{align*}

Let us now split these sets $MPTT_\star(\mu)$ depending on the degree of the marked vertex:
\[ MPTT_\star(\mu) = \bigsqcup_k MPTT_\star^k(\mu), \]
where $MPTT_\star^k(\mu)$ denote the subset of $MPTT_\star(\mu)$ of
trees with a marked vertex of degree $k$. By Lemma \ref{LemFin1} (see next paragraph), one has:
\[ \text{for all } k \geq 1,\ \big|MPTT_1^k(\mu)\big| = 
\frac{m_k(\mu)}{p} \big|MPTT_1(\mu)\big|. \] 

Let us consider an element of $MPTT_1^k(\mu)$. We distinguish two cases:
\begin{itemize}
  \item either the end of the edge leaving $\overline{\pi}$ in the graph
      $G(\tau,\sigma)$ is $\overline{\pi}$ itself.
      In this case, the graph $G(\tau,\sigma)$ contains a loop
      and the element is not in $MPTT_{1,2}^k(\mu)$.
  \item or it is another vertex of the tree.
      We call such marked permuted thorn trees \emph{good}.
      We will prove below (Lemma \ref{LemFin3}) with the induction hypothesis
      that, in this case, exactly $n-1$
      elements over $(p-1)(n-p+1)$ are in $MPTT_{1,2}^k(\mu)$.
\end{itemize}
By Lemma \ref{LemFin2}, the second case concerns exactly $n-k$ elements over $n-1$.
Therefore:
\[ | MPTT_{1,2}^k(\mu) | = \frac{n-1}{(p-1)(n-p+1)} \left( \frac{n-k}{n-1} \big| MPTT_1^k(\mu) \big| \right)\]
and we can compute $P'(\mu)$ as follows

\begin{align*}
P'(\mu) & = \frac{\big|MPTT_{1,2}(\mu)\big|}{\big|MPTT_{1}(\mu)\big|}
=\frac{\sum_k \big|MPTT_{1,2}^k(\mu)\big|}{ \big|MPTT_{1}(\mu)\big|}\\
P'(\mu) &= \frac{\sum_k \frac{n-k}{(p-1)(n-p+1)}
\big| MPTT_1^k(\mu) \big|}{\big|MPTT_{1}(\mu)\big|}\\
P'(\mu) &= \frac{\sum_k \frac{n-k}{(p-1)(n-p+1)}
\frac{m_k(\mu)}{p} \big| MPTT_1(\mu) \big|}{\big|MPTT_{1}(\mu)\big|}\\
P'(\mu) & = \frac{1}{(p-1)\big(n-p+1\big)} \cdot \left[ \frac{1}{p} \left( \sum_k n \cdot m_k(\mu) - k \cdot m_k(\mu) \right) \right] ;\\
P'(\mu) & = \frac{1}{(p-1)\big(n-p+1\big)} \frac{n\cdot p - n}{p} ;\\
P'(\mu) &
=\frac{n}{p\big(n-p+1\big)}.
\end{align*}
This computation ends the proof of Proposition \ref{prop:proportion_P2} and,
therefore, of equation \eqref{EqDBT}.
\end{proof}

\subsection{Technical lemmas}
Let $\mu$ be a partition of size $n$ and length $p$.

\begin{Lem}
\label{LemFin1}
For all $k \geq 1$,
\[ \big|MPTT_1^k(\mu)\big| = 
\frac{m_k(\mu)}{p} \big|MPTT_1(\mu)\big|. \] 
\end{Lem}
\begin{proof}
Consider the action of $S_p$ on $PTT_1(\mu)$ consisting in permuting the black vertices (with their thorns). In each orbit and hence in the whole set $PTT_1(\mu)$, the proportion of elements for which the left-most black vertex $\pi_0$ has degree $k$ is $\frac{m_k(\mu)}{p}$.
To each element in $PTT_1(\mu)$ correspond exactly $p-1$ elements in $MPTT_1(\mu)$
obtained by choosing a marked vertex $\overline{\pi}$ among the black vertices different from $\pi_0$.
Therefore the probability that $\overline{\pi}$ has degree $k$ is also $\frac{m_k(\mu)}{p}$,
which is what we wanted to prove.

Note that this is not true if we consider elements with property $(P2)$ as the action of $S_p$ does not preserve this property.
\end{proof}

We denote by $GMPTT_1^k(\mu)$ the set of good
marked permuted thorn trees $(\tau,\sigma,\overline{\pi})$ of type $\mu$, 
for which $\overline{\pi}$ is a vertex of degree $k$.

\begin{Lem}
\label{LemFin2}
\[\frac{|GMPTT_1^k(\mu)|}{|MPTT_1^k(\mu)|} = \frac{n-k}{n-1}. \]
\end{Lem}
\begin{proof}
    Consider the action of $S_{n-1}$ on $MPTT_1^k(\mu)$ consisting in changing
    the cyclic order of the edges and thorns incident to the root without moving
    the left-most edge.
    In each orbit of this action, the edge or thorn $e'$ just after
    $e_{\overline{\pi}}$ is uniformly distributed among the $n-1$ edges and
    thorns incident to the root and different from $e_{\overline{\pi}}$.
    Among these edges and thorns, there are $k-1$ thorns which are associated
    by $\sigma$ to a thorn incident to the black vertex $\overline{\pi}$.
    By definition, an element in $MPTT_1^k(\mu)$ is good if and only if $e'$ is not one of
    these thorns, therefore, in each orbit, the proportion of good elements
    is $\frac{n-k}{n-1}$.
\end{proof}

Recall that any marked permuted thorn tree verifying property $(P2)$
is good. In other terms, $MPTT_{1,2}^k(\mu)$ is a subset of $GMPTT_1^k(\mu)$.
\begin{Lem}
\label{LemFin3}

We assume that, for $\mu'$ of size $n-1$ and length $p-1$,
the proportion of permuted star thorn trees of type $\mu'$ verifying $(P2)$
among those which verify $(P1)$ does not depend on $\mu'$.
We denote this proportion $P'_{n-1,p-1}$.
Then one has:
\[\frac{|MPTT_{1,2}^k(\mu)|}{|GMPTT_1^k(\mu)|} = P'_{n-1,p-1}.\]
\end{Lem}

\begin{proof}
    Consider the following application
    \[ \varphi_{\mu,k} : \begin{array}{rcl}
        GMPTT_1^k(\mu) &\longrightarrow& \left\{ \begin{tabular}{c}
            permuted star thorn trees with \\
            $\ell(\mu)-1$ black vertices and $n- \ell(\mu)$ thorns
        \end{tabular} \right\} \\
        (\tau,\sigma,\overline{\pi}) &\longmapsto& (\tau',\sigma'),
    \end{array}\]
    where $(\tau',\sigma')$ is obtained as follows.
    Consider the edge or thorn immediately to the left of $e_{\overline{\pi}}$ and
    denote $\pi'$ the black extremity of the element with the same symbolic label.
    Then, starting from $(\tau,\sigma,\overline{\pi})$,
    erase the marked black vertex $\overline{\pi}$ with its edge $e_{\overline{\pi}}$
    and move its thorns to the black vertex $\pi'$
    (at the right of its own thorns). For example,
    \[ \varphi_{\mu,k} \left( \begin{array}{c}
       \includegraphics[width=35mm]{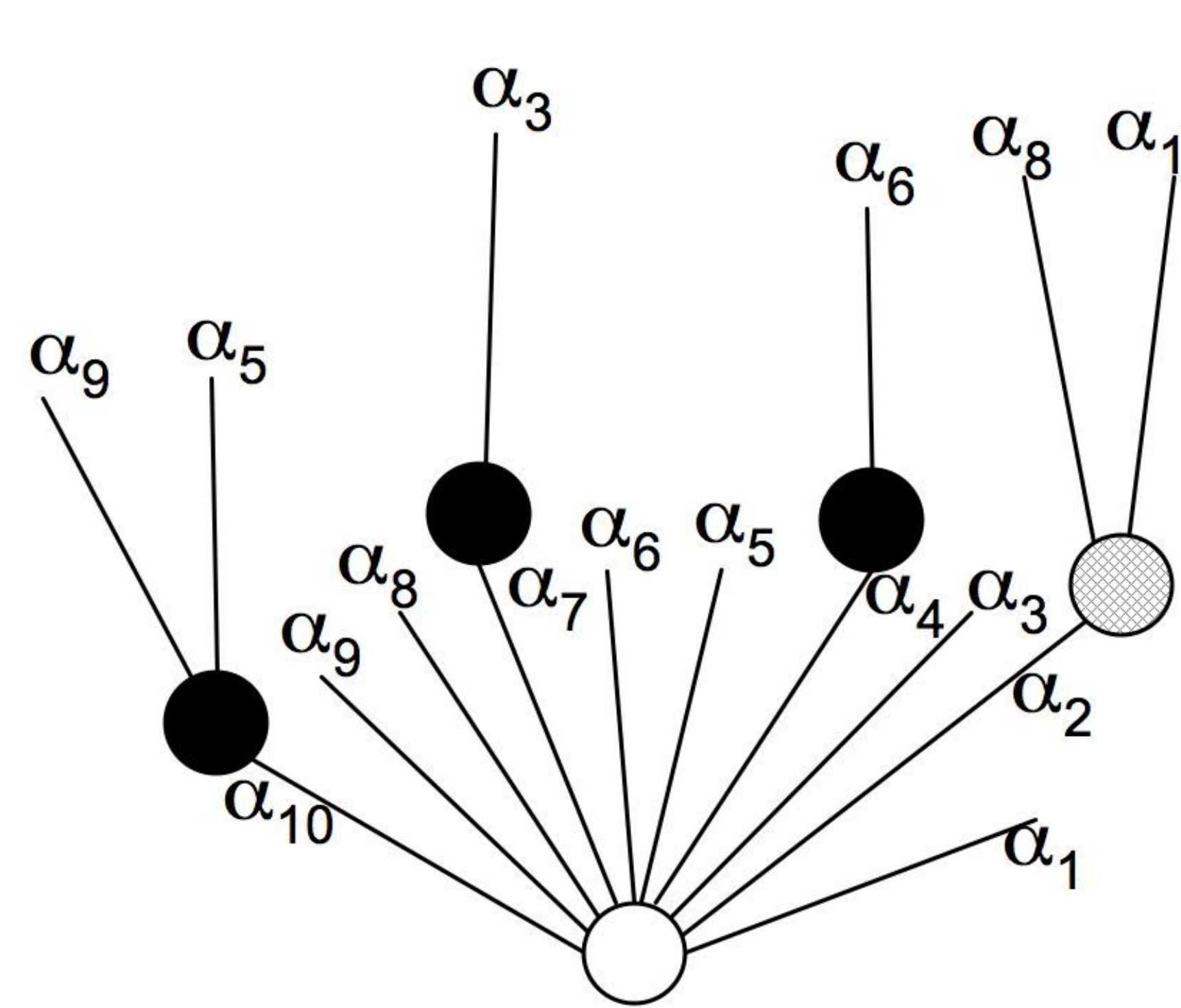}\end{array} \right) =
       \begin{array}{c}\includegraphics[width=35mm]{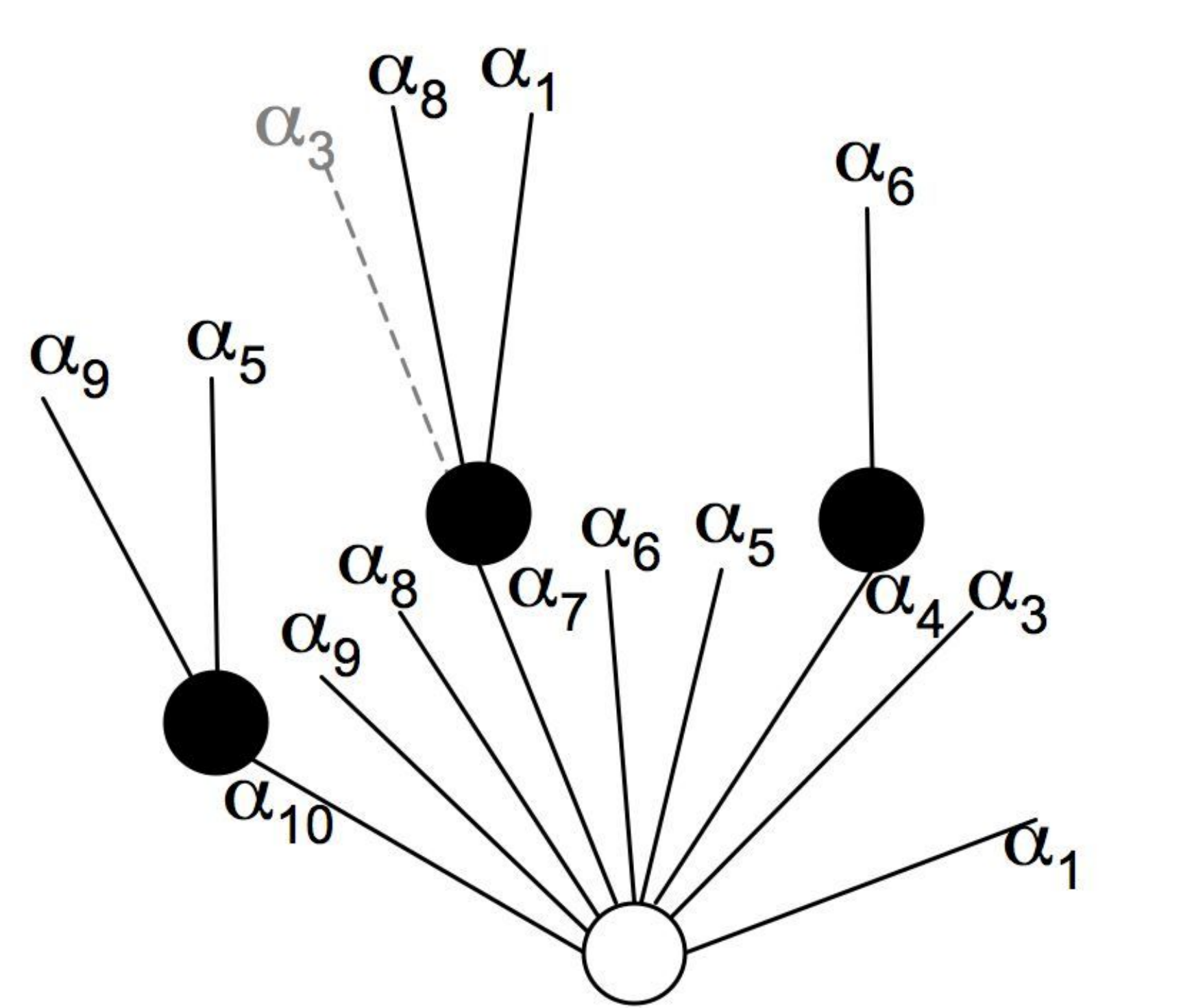} \end{array}.
            \]
    This application has nice properties:
    \begin{itemize}
        \item it preserves property $(P2)$.
            Indeed, if $(\tau',\sigma') = \varphi(\tau,\sigma,\overline{\pi})$,
            then $G_{\tau',\sigma'}$ is obtained form $G_{\tau,\sigma}$ by 
            contracting its edge attached to the vertex $\overline{\pi}$.
        \item the number of preimages of a given permuted star thorn tree
            $(\tau',\sigma')$ depends only on its type $\lambda$. Indeed, there are no preimages if
            $\lambda$ is not of the form
            $\mu \backslash (j,k) \cup (j+k-1)$
            for some $j$
            (from now on, we use the notation $\mu^{\downarrow (j,k)} = \mu \backslash (j,k) \cup (j+k-1)$).
            Otherwise, the preimages are obtained as follows: choose a vertex $v$
            of $\tau'$ of degree $j+k-1$ (there are $m_{j+k-1}(\lambda)$ possible choices),
            choose the edge or a thorn of white
            extremity associated to one of its $j-1$ left-most thorns
            ($j$ choices per vertex $v$), add
            a new black vertex just at the right of this element and attach the
            $k-1$ last thorns of $v$ to this new vertex.
            With this description, it is clear that 
            the cardinality of preimage is $j m_{j+k-1}(\lambda)$.
    \end{itemize}
    Recall that we assumed the number $P'_{n-1,p-1}$ ( dependent only on $n$ and $p$, but not on $\lambda$) to be the proportion of permuted star thorn trees of type $\lambda$ verifying $(P2)$ among those which verify $(P1)$.
    With the two above properties, we can compute the proportion of elements
    verifying $(P2)$ in $GMPTT_1^k(\mu)$. Indeed,
    \begin{multline*}
    |MPTT_{1,2}^k(\mu)|
    =\sum_{j \geq 1 \atop \lambda=\mu^{\downarrow (j,k)}}
    j m_{j+k-1}(\lambda) |MPTT_{1,2}^k(\lambda)| \\
    =  \sum_{j \geq 1 \atop \lambda=\mu^{\downarrow (j,k)}}
    j m_{j+k-1}(\lambda) P'_{n-1,p-1} |MPTT_{1}^k(\lambda)| 
        = P'_{n-1,p-1} |GMPTT_{1}^k(\mu)|,
    \end{multline*}
    which is exactly what we wanted to prove.
\end{proof}
\section{Link between Theorems \ref{th:reformulationfine} and \ref{th:Stanleyfine}}\label{sect:reformulation}
The goal of this section is to prove the equivalence between
Theorem \ref{th:reformulationfine} and Theorem \ref{th:Stanleyfine}.
This will be done using differential calculus in the symmetric function ring :
we present this algebra in paragraph \ref{SubsectSymmetricFunctions}.
Then, in paragraph \ref{subsect:LinkOGF}, we explain how the
generating series of black-partitioned maps and maps are related.
Finally, after a small lemma on thorn trees (paragraph \ref{subsect:lemma_thorn_trees}),
we use all these tools to prove the equivalence of
Theorems \ref{th:reformulationfine} and \ref{th:Stanleyfine} in paragraph \ref{SubsectEquivalence}.

\subsection{Symmetric functions}\label{SubsectSymmetricFunctions}
Let us begin by some definitions and notations on symmetric functions. As much as possible we use the notations
of I.G. Macdonald's book \cite{Macdo}.

We consider the ring $\Lambda_n$ of symmetric polynomials in $n$ variables
$x_1,\dots,x_n$.
The sequence $(\Lambda_n)_{n \geq 1}$ admits a projective limit $\Lambda$,
called \emph{ring of symmetric functions}.
This ring has several classical linear bases indexed by partitions.
\begin{itemize}
    \item \emph{monomial symmetric functions}:
        for monomials we use  the short notation
        $\xx^\vv=x_1^{v_1} x_2^{v_2} \dots$.
        Then, we define        
        \[\mon_\lambda = \sum_{\vv} \xx^\vv\]
        where the sum runs over all vectors $\vv$ which are permutations of $\lambda$
        (without multiplicities).

        \noindent {\it Remark.} We use upper case $M$ for the monomial symmetric
        functions instead of the usual lower case $m$ because a lot of formulae 
        in this paper involve
        multiplicities $m_i(\lambda)$ of some parts and monomial symmetric functions
        at the same time.
    \item \emph{power sum symmetric functions}:
       by definition 
        \[p_0 =1,\qquad p_k =\sum_{i \geq 1} x_i^k,
        \qquad p_\mu =\prod_{j=1}^{\ell(\mu)} p_{\mu_j}.\]
\end{itemize}

Besides, we consider the differential operator $\Delta_n : \Lambda_n \to \Lambda_n$ given by:
\[ \text{for all }f \in \Lambda_n, \Delta_n ( f ) = \sum_{i=1}^n x_i^2 \frac{\partial f}{\partial x_i}. \]

Let us compute the image by this operator of the symmetric polynomials
$\mon_\lambda(x_1,\dots,x_n)$ and $p_\mu(x_1,\dots,x_n)$.
If $\ell(\lambda) \leq n$, we denote $S_n(\lambda)$ the set (without multiplicities)
of all vectors obtained by a permutation of the vector 
$(\lambda_1,\dots,\lambda_{\ell(\lambda)},0,\dots,0)$ of size $n$.
\begin{align*}
 \Delta_n\big(\mon_\lambda(x_1,\dots,x_n)\big) &=
 \sum_{\vv \in S_n(\lambda)} \sum_{i=1}^n x_i^2 \frac{\partial \xx^\vv}{\partial x_i},\\
 &=  \sum_{\vv \in S_n(\lambda)} \sum_{i=1}^n  v_i \xx^{\vv + \delta_i},
\end{align*}
where $\delta_i$ is the vector of length $n$, whose components are all equal to $0$,
except for its $i$-th component which is equal to $1$. It is clear that, if $\vv$ is a permutation
of a partition $\lambda$, then $\vv + \delta_i$ is a permutation of some
$\mu= \lambda^{\uparrow (j)}$ for $j=v_i$.

We will group together terms with the same exponent.
So the question is: given a vector $\vv'$, which is a
permutation of $\mu$, in how many ways can it be written as $\vv + \delta_i$ with
$\vv \in S_n(\lambda)$ and $1\leq i \leq n$? The vector $\vv' - \delta_i$ is a permutation
of $\mu^{\downarrow (v'_i)}$, which is equal to $\lambda$ if and only if $v'_i=j+1$. Therefore,
there are $m_{j+1}(\mu)$ ways to write $\vv'$ under this form. Finally,
\begin{align*}
 \Delta_n\big(\mon_\lambda(x_1,\dots,x_n)\big) 
 &= \sum_{j>0 \atop \mu=\lambda^{\uparrow (j)}}  \sum_{\vv' \in S_n(\mu)} j \cdot m_{j+1}(\mu) \xx^{\vv'}\\
&= \sum_{j>0 \atop \mu=\lambda^{\uparrow (j)}} j \cdot  m_{j+1}(\mu)\ \mon_\mu(x_1,\dots,x_n).
\end{align*}
As the coefficients in this formula do not depend on $n$,
one can define the limit of the operators $\Delta_n$
as the operator $\Delta : \Lambda \to \Lambda$ which sends $\mon_\lambda$ to 
\begin{equation}\label{EqDeltaMon}
\Delta(\mon_\lambda) = \sum_{j>0 \atop \mu=\lambda^{\uparrow (j)}} j \cdot m_{j+1}(\mu)\ \mon_\mu.
\end{equation}
It is the limit of the sequence $(\Delta_n)_{n \geq 1}$ in the sense that:
\[ \text{for all } F \in \Lambda,\ (\Delta F)(x_1,\dots,x_n) = \Delta_n \big( F(x_1,\dots,x_n) \big).\]
Note that it was not obvious before the computation that
the sequence of ope\-rators $\Delta_n$ had a limit.
For instance, the sequence of operators  $\Delta'_n$ defined by
 $\Delta'_n(f)=\sum_{i=1}^n \frac{\partial f}{\partial x_i} $ does not
have a limit because $\Delta'_n \big( \mon_{(1)} (x_1,\dots,x_n) \big)= n$
does not have a limit in $\Lambda$.

Let us now come to the image of power sums.
For one part partition, one has, for $k \geq 1$:
\[
\Delta_n \big( p_k(x_1,\dots,x_n) \big)
= \sum_{1 \leq i,j \leq n} x_i^2 \frac{\partial x_j^k}{\partial x_i}
= \sum_{1 \leq i \leq n} k \cdot x_i^{k+1} = k \cdot p_{k+1} (x_1,\dots,x_n).
\]
The result still holds for $k=0$. Using the fact that  $\Delta_n$ is a derivation,
one obtains immediately the formula for general power sums:
\begin{align*}
\Delta_n \big( p_\lambda(x_1,\dots,x_n) \big) & = \sum_j \left( \lambda_j \cdot  p_{\lambda_j+1}(x_1,\dots,x_n) \cdot \prod_{\ell \neq j} p_\ell \right) \\
& = \sum_i i \cdot m_i(\lambda)\ p_{\lambda^{\uparrow (i)}}(x_1,\dots,x_n).
\end{align*}
One can take the limit of the previous equation and we get:
\begin{equation}\label{EqDeltaPower}
 \Delta(p_\lambda) = \sum_i i \cdot m_i(\lambda)\ p_{\lambda^{\uparrow (i)}}.
\end{equation}

\subsection{Generating series of maps and partitioned maps}\label{subsect:LinkOGF}

Recall that $A(\lambda)$, $B(\lambda)$, $C(\lambda)$ and $D(\lambda)$ count the numbers
of (star) (partitioned) rooted unicellular bipartitite maps of type $\lambda$,
according to the following table.

\begin{tabular}{c|c|c}
    & \begin{tabular}{c} maps without \\ additional structure \end{tabular} 
        & \begin{tabular}{c} partitioned \\ maps \end{tabular} \\ \hline
            \begin{tabular}{c} no conditions \\ on white vertices \end{tabular}
                & $A(\lambda)$ & $C(\lambda)$  \\ \hline
                \begin{tabular}{c} only one \\ white vertex \end{tabular}
                    & $B(\lambda)$ & $D(\lambda)$
\end{tabular}\medskip

Quantities $A$ and $C$ (resp. $B$ and $D$) are linked by the following lemma:
\begin{Lem}\label{LemLinkOGF}
\begin{align}
\label{eq:C2A} \sum_{\mu \vdash n+1} C(\pnuv\mu) \Aut(\mu) \mon_\mu & = \sum_{\nu \vdash n+1} A(\nu) p_\nu;\\
\label{eq:D2B} \sum_{\lambda \vdash n} D(\pnv\lambda) \Aut(\lambda) \mon_\lambda & = \sum_{\pi \vdash n} B(\pnv\pi) p_\pi,
\end{align}
where $\Aut(\mu)$ is the numerical factor $\prod_i m_i(\mu)!$ by definition.
\end{Lem}
\begin{proof}

The proof is similar to the one of \cite[Proposition 1]{MoralesVassilieva:factorizations_long_cycle}. We note $\overline{R}_{\epsilon,\rho}$ the number of ways to {\it coarse} an integer partition $\epsilon \vdash n$ to get an integer partition $\rho$, i.e. the number of unordered set partitions $\{P^{1}, \ldots, P^{\ell(\rho)}\}$ of $[\ell(\epsilon)]$ such that $\rho_j = \sum_{i\in P^j} \epsilon_i$. We have the classical relation:  $p_{\epsilon} = \sum_{\rho} Aut(\rho) \overline{R}_{\epsilon,\rho} \mon_{\rho}$.\\
 Furthermore by definition of partitioned maps, $C(\mu) = \sum_{\nu}\overline{R}_{\nu,\mu}A(\nu)$ (resp. $D(\lambda) = \sum_{\pi}\overline{R}_{\pi,\lambda}B(\pi)$). Combining these expressions yields the desired result.\end{proof}

 \subsection{An easy lemma on permuted thorn trees}\label{subsect:lemma_thorn_trees}
 Consider integers $n,i \geq 1$ and two partitions $\lambda \vdash n$,
 $mu \vdash n+1$ with $\mu=\lambda^{\uparrow (i)}$.

It is easy to transform a permuted thorn tree $(\tau,\sigma)$ where $\tau$ has type $\lambda \vdash n$ into a permuted thorn tree $(\tau',\sigma')$ where $\tau'$ has type $\mu$.
We just add a thorn anywhere on the white vertex ($n+1$ possible places) and a thorn anywhere on a black vertex of degree $i$ (there are $i$ possible places on each of the $m_i(\lambda)$ black vertices of degree $i$). Then we choose $\sigma'$ to be the extension of $\sigma$ associating the two new thorns. This procedure is invertible if we remember which thorn of black extremity is the new one (it must be on a black vertex of degree $i+1$, so there are $i \cdot m_{i+1}(\mu)$ choices). This leads immediately to the following relation:
\begin{equation}\label{eq:rec_BT}
\BT(\pnuv\mu) \cdot (n+1-p)! \cdot i \cdot m_{i+1}(\mu) = (n+1) \cdot i \cdot m_i(\lambda) \cdot \BT(\pnv\lambda) \cdot (n-p)!.
\end{equation}
If we fix a partition $\mu \vdash n+1$ of length $p<n+1$ and
 sum equation \eqref{eq:rec_BT} over partitions $\lambda$ that write
 as $\mu^{\downarrow (i+1)}$ for some $i$, we get:
\begin{equation}\label{EqSumRecBT}
\BT(\pnuv\mu) \cdot (n+1-p)! \cdot (n+1-p) = (n+1) \sum_{\lambda = \mu^{\downarrow (i+1)}, i>0} i \cdot m_i(\lambda) \cdot \BT(\pnv\lambda) \cdot (n-p)!.
\end{equation}

\subsection{Counting partitioned or not partitioned maps are equivalent} \label{SubsectEquivalence}

We have now all the tools to prove the equivalence of Theorems
\ref{th:reformulationfine} and \ref{th:Stanleyfine}.

\begin{proof}
Let us first assume that Theorem \ref{th:reformulationfine},
and hence equation \eqref{EqDBT}, is true.

We start from equation \eqref{EqSumRecBT} and use equations \eqref{eq:bij_MV}
and \eqref{EqDBT} respectively in the left and right-hand sides: for any $\mu \vdash n+1$,
\begin{align}
C(\pnuv\mu) \cdot (n+1-p) & = (n+1) \sum_{\lambda = \mu^{\downarrow (i+1)}, i>0} i \cdot m_i(\lambda) \cdot D(\pnv\lambda) \cdot (n+1-p).\label{EqProofEquiv1}\\
\intertext{We multiply both sides by $\Aut(\mu) \mon_\mu$ and sum this equality on all partitions $\mu$ of $n+1$, except $1^{n+1}$.}
\sum_{\mu \vdash n+1 \atop \mu \neq 1^{(n+1)}} C(\pnuv\mu) \Aut(\mu) \mon_\mu & = (n+1) \sum_{\mu \vdash n+1 \atop \mu \neq 1^{(n+1)}} \sum_{i>0 \atop \lambda = \mu^{\downarrow (i+1)}} i \cdot m_i(\lambda) \Aut(\mu) D(\pnv\lambda) \mon_\mu  \label{EqThReformulationMSeries}\\
& = (n+1) \sum_{\lambda \vdash n} \Aut(\lambda) D(\pnv\lambda) \left(\sum_{i>0 \atop \mu = \lambda^{\uparrow (i)}} i \cdot m_{i+1}(\mu) \mon_\mu \right). \nonumber
\end{align}
The last equality has been obtained by changing the order of summation and using the trivial fact that, if $\mu = \lambda^{\uparrow (i)}$, one has $\Aut(\mu) \cdot m_i(\lambda)= \Aut(\lambda) \cdot m_{i+1}(\mu)$. Now, observing that the expression in the brackets is exactly the right hand-side of equation \eqref{EqDeltaMon}, one has:
$$\sum_{\mu \vdash n+1} C(\pnuv\mu) \Aut(\mu) \mon_\mu - (n+1)! M_{1^{n+1}} = (n+1) \cdot \Delta \left( \sum_{\lambda \vdash n} \Aut(\lambda) D(\pnv\lambda) \mon_\lambda \right).$$
Let us rewrite this equality in the power sum basis. The expansion of the two summations in this basis are given by equations \eqref{eq:C2A} and \eqref{eq:D2B}. We also need the power sum expansion of $(n+1)! M_{1^{n+1}}$, which is (see \cite[Chapter I, equation (2.14')]{Macdo}):
\begin{multline*}
(n+1)! M_{1^{n+1}} = (n+1)! \sum_{\nu \vdash n+1} \frac{(-1)^{n+1-\ell(\nu)}}{z_\nu} p_\nu
 = \sum_{\nu \vdash n+1} A(\nu) (-1)^{n+1-\ell(\nu)} p_\nu.
\end{multline*}
Putting everything together, we get:
\begin{multline}\label{EqThMainPSeries}
\sum_{\nu \vdash n+1} A(\nu) p_\nu + \sum_{\nu \vdash n+1} A(\nu) (-1)^{n-\ell(\nu)} p_\nu = (n+1) \sum_{\pi \vdash n} B(\pnv\pi) \Delta(p_\pi)\\
= (n+1) \sum_{\pi \vdash n} B(\pnv\pi) \sum_i i \cdot m_i(\pi)\ p_{\pi^{\uparrow (i)}}.
\end{multline}
The last equality comes from equation \eqref{EqDeltaPower}.
Identifying the coefficients of $p_\mu$ in both sides, we obtain exactly Theorem \ref{th:Stanleyfine}.
\bigskip

Conversely, let us suppose that Theorem \ref{th:Stanleyfine} is true.
This means that, for every partition $\mu \vdash n+1$, one has:
\[A(\nu) + (-1)^{n-\ell(\nu)} A(\nu) = (n+1) 
\sum_{\pi = \nu^{\downarrow (i+1)}, i>0}  i\ m_i(\pi)\ B(\pnv\pi).\]
Multiplying by $p_\mu$ and summing over all partitions $\mu$ of $n+1$,
we obtain equation \eqref{EqThMainPSeries}.
With the same computations as before we can deduce equation 
\eqref{EqThReformulationMSeries} from it .
Identifying the coefficients of $M_\mu$ , we get equation \eqref{EqProofEquiv1}.
But, using equations \eqref{eq:bij_MV} and \eqref{EqSumRecBT}, one has:
\begin{multline*}
C(\pnuv\mu) \cdot (n+1-p) = \BT(\pnuv\mu) \cdot (n+1-p)! \cdot (n+1-p)\\
 = (n+1) \sum_{\lambda = \mu^{\downarrow (i+1)}, i>0} i \cdot m_i(\lambda) \cdot \BT(\pnv\lambda) \cdot (n-p)!.
\end{multline*}
Therefore, for every $\mu \vdash n+1$,
\[ \sum_{\lambda = \mu^{\downarrow (i+1)}, i>0}
i \cdot m_i(\lambda) \cdot \BT(\pnv\lambda) \cdot (n-p)!
= \sum_{\lambda = \mu^{\downarrow (i+1)}, i>0}
i \cdot m_i(\lambda) \cdot D(\pnv\lambda) \cdot (n+1-p).\]
Using Remark \ref{rem:syst_inv} implies that for any $\lambda \vdash n$
\[\BT(\pnv\lambda) \cdot (n-p)! = D(\pnv\lambda) \cdot (n+1-p),\]
because both sides are solutions of the same sparse triangular system.
This corresponds to equation \eqref{EqDBT}, one of the equivalent forms of
Theorem~\ref{th:reformulationfine}.\end{proof}

\begin{Rem}
    Using the same kind of arguments, one could also prove that
    Theorem \ref{th:stanley} and Theorem~\ref{th:reformulation} are
    equivalent.
\end{Rem}

%%%%%%%%%%%%%%%%%%%%%%%%%%%%%%%%%%%%%%%%%%%%%%%%
%%%%%%%%%      new section     %%%%%%%%%%%%%%%%%
%%%%%%%%%%%%%%%%%%%%%%%%%%%%%%%%%%%%%%%%%%%%%%%%

\section*{Acknowledgements}
The second author acknowledges the support of ERC under the agreement  "ERC StG 208471 - ExploreMaps".

%\section*{References}

\bibliographystyle{alpha}
\bibliography{product_of_long_cycles}

\end{document}